\documentclass{siamltex}  
\usepackage[latin1]{inputenc}
\usepackage{mgaproof}
\usepackage{cite}
\usepackage{graphicx}
\usepackage{amsmath,amssymb,amsfonts,amsxtra,mathrsfs,bm}
\usepackage[ruled,vlined]{algorithm2e}
\usepackage{multicol}
\usepackage{multirow}
\usepackage{paralist}
\usepackage{xspace}

\newcommand{\Xc}{\mathcal{X}}
\newcommand{\Gc}{\mathcal{G}}
\newcommand{\Fc}{\mathcal{F}}

\newcommand{\Hc}{\mathcal{H}}

\newcommand{\smet}{\delta_s}

\newcommand{\ds}{\delta_S}
\newcommand{\inv}[1]{#1^{-1}}
\newcommand{\invp}[1]{\left(#1\right)^{-1}}
\newcommand{\nlsum}{\sum\nolimits}
\newcommand{\nlmin}{\min\nolimits}
\newcommand{\nlprod}{\prod\nolimits}
\newcommand{\reals}{\mathbb{R}}
\newcommand{\C}{\mathbb{C}}

\newcommand{\posdef}{\mathbb{P}}
\newcommand{\pd}{\posdef}
\renewcommand{\H}{\mathbb{H}}
\newcommand{\riem}{\delta_R}
\newcommand{\jsld}{S}
\newcommand{\thom}{\delta_T}
\newcommand{\hilb}{\delta_H}
\newcommand{\ssd}{\delta_S}
\newcommand{\gm}{{\sharp}}

\newcommand{\sdiv}{S-Divergence\xspace}
\newcommand{\majlw}{\prec_{w\log}}
\newcommand{\majl}{\prec_{\log}}
\newcommand{\majw}{\prec_w}

\newcommand{\ld}{\ell{}d}
\newcommand{\vn}{vn}

\newcommand{\da}{\downarrow}

\newcommand{\pfrac}[2]{\left(\tfrac{#1}{#2}\right)}

\newcommand{\ip}[2]{\langle {#1},\, {#2} \rangle}
\newcommand{\frob}[1]{\|{#1}\|_{\text{F}}}

\newcommand{\mynorm}[2]{\| {#1} \|_{#2}}
\newcommand{\norm}[1]{\mynorm{#1}{}}
\newcommand{\pnorm}[2]{\mynorm{#1}{#2}}

\newcommand{\norml}[1]{\mynorm{#1}{1}}

\newcommand{\infnorm}[1]{\mynorm{#1}{\infty}}

\newcommand{\enorm}[1]{\mynorm{#1}{}}

\newcommand{\half}{\tfrac{1}{2}}

\newcommand{\set}[1]{\left\{ {#1}\right\}}

\newcommand{\kron}{\otimes}

\DeclareMathOperator*{\argmin}{argmin}

\DeclareMathOperator{\trace}{tr}
\DeclareMathOperator{\Diag}{Diag}

\DeclareMathOperator{\Eig}{Eig}
\DeclareMathOperator{\mydet}{det}

\newcommand{\citep}{\cite}
\newcommand{\citet}{\cite}
\newtheorem{remark}[theorem]{Remark}
\newtheorem{example}[theorem]{Example}

\begin{document}

\title{Positive definite matrices and the S-divergence\thanks{A small fraction of an \emph{initial version of this work} was presented at the \emph{Advances in Neural Information Processing Systems (NIPS) 2012} conference}---see \citep{srapd}.}

\author{Suvrit Sra\thanks{Parts of this paper were written during my stay at Carnegie Mellon University; the initial version was prepared while I was at the Max Planck Institute for Intelligent Systems, T\"ubingen, Germany.}}  

\maketitle

\begin{abstract}
  Positive definite matrices abound in a dazzling variety of applications. This ubiquity can be in part attributed to their rich geometric structure: positive definite matrices form a self-dual convex cone whose strict interior is a Riemannian manifold. The manifold view is endowed with a ``natural'' distance function while the conic view is not. Nevertheless, drawing motivation from the conic view, we introduce the \emph{S-Divergence} as a ``natural'' distance-like function on the open cone of positive definite matrices. We motivate the S-divergence via a sequence of results that connect it to the Riemannian distance. In particular, we show that (a) this divergence is the square of a distance; and (b) that it has several geometric properties similar to those of the Riemannian distance, though without being computationally as demanding. The \sdiv is even more intriguing: although nonconvex, we can still compute matrix means and medians using it to global optimality. We complement our results with some numerical experiments illustrating our theorems and our optimization algorithm for computing matrix medians. 
\end{abstract}


\section{Introduction}
Hermitian positive definite (HPD) matrices are a noncommutative generalization of positive reals. They abound in a multitude of applications and exhibit attractive geometric properties---e.g., they form a differentiable Riemannian (also  Finslerian) manifold~\citep{bhatia07,hiai} that is a well-studied example of a manifold of nonpositive curvature \citep[Ch.10]{bridson}. HPD matrices possess even more structure: (i) they embody a canonical higher-rank symmetric space \citep{terras}; and (ii) their closure forms a closed, self-dual convex cone.

The convex conic view enjoys great importance in convex optimization \citep{ipbook,lemco,neTo02} and in nonlinear Perron-Frobenius theory~\citep{lemNuss12}; symmetric spaces are important in algebra, analysis \citep{terras,helgason,leeLim}, and optimization \citep{ipbook,sdps}; while the manifold view (Riemannian or Finslerian) plays diverse roles---see~\citep[Ch.6]{bhatia07} and \cite{nieBha13}.

The manifold view is equipped with a with a ``natural'' distance function while the conic view is not. Nevertheless, drawing motivation from the convex conic view, we introduce the \emph{S-Divergence} as a ``natural'' distance-like function on the open cone of positive definite matrices. Indeed, we prove a sequence of results connecting the \sdiv to the Riemannian distance. Most importantly, we show that (a) this divergence is the square of a distance; and (b) that it has several geometric properties in common with the Riemannian distance, without being numerically as demanding. This builds an informal link between the manifold and conic views of HPD matrices. 

\subsection{Background and notation}
We begin by fixing notation. The letter $\Hc$ denotes some Hilbert space, usually just $\C^n$. The inner product between two vectors $x$ and $y$ in $\Hc$ is $\ip{x}{y} := x^*y$ ($x^*$ denotes `conjugate transpose'). The set of $n\times n$ Hermitian matrices is denoted as $\H_n$. A matrix $A \in \H_n$ is called \emph{positive definite} if
\begin{equation}
  \label{eq.pd}
  \ip{x}{Ax} > 0\quad\text{for all}\quad x \neq 0,\qquad\text{also written as}\quad A > 0.
\end{equation}
The set of all positive definite (henceforth \emph{positive}) matrices is denoted by $\pd_n$. We say $A$ is \emph{positive semidefinite} if $\ip{x}{Ax} \ge 0$ for all $x$; denoted $A \ge 0$. The inequality $A \ge B$ is the usual L\"owner order and means $A-B \ge 0$. The \emph{Frobenius norm} of a matrix $X \in \C^{m\times n}$ is defined as $\frob{X} = \sqrt{\trace(X^*X)}$; and $\enorm{X}$ denotes the operator 2-norm. Let $f$ be an analytic function on $\C$; for a matrix $A$ with eigendecomposition $A=U\Lambda U^*$, $f(A)$ equals $Uf(\Lambda)U^*$ with $f(\Lambda)=\Diag[f(\lambda_1),\ldots,f(\lambda_n)]$.

The set $\pd_n$ is a well-studied differentiable Riemannian manifold, with the Riemannian metric given by the differential form $ds = \frob{A^{-1/2}dAA^{-1/2}}$. This metric induces the \emph{Riemannian distance}~(see e.g.,~\cite[Ch.~6]{bhatia07}):
\begin{equation}
  \label{eq.riem}
  \riem(X,Y) := \frob{\log(Y^{-1/2}XY^{-1/2})}\quad\text{for}\quad X, Y > 0,
\end{equation}
and where $\log(\cdot)$ denotes the matrix logarithm. 

A counterpart to the distance~\eqref{eq.riem} was formally introduced in~\citep{srapd} under the name \emph{\sdiv}\footnote{It is a divergence because although nonnegative,
  definite, and symmetric, it is \emph{not} a metric.}; this divergence is defined as
\begin{equation}
  \label{eq.ss}
  \ds^2(X,Y) := \log\det\left(\frac{X+Y}{2}\right) - \frac{1}{2}\log\det(XY)\quad\text{for}\quad X, Y > 0.
\end{equation}
Our definition above already writes $\ds^2$ in anticipation of Theorem~\ref{thm.metric} that shows $\ds$ to be a metric. This paper suggests \sdiv as an alternative to~\eqref{eq.riem}, and studies several of its properties that may also be of independent interest. The simplicity of~\eqref{eq.ss} is one of the key reasons for using it as an alternative to~\eqref{eq.riem}: it is cheaper to compute, as is its derivative, and certain basic algorithms involving it run much faster than corresponding ones that use $\riem$~\citep{srapd}.

This line of thought actually originates in~\cite{iccv11,chSra12}, where for an image search task based on ``nearest neighbors,'' $\ds^2$ is used to measure nearness instead of $\riem$, and is shown to yield large speedups without blighting the quality of search results. Although exact details of this image search are outside the scope of this paper, let us highlight below the two speedups that were crucial to~\cite{iccv11,chSra12}. 

The first speedup is shown in the left panel of Fig.~\ref{fig.one}, which compares times taken to compute $\ds^2$ and $\riem$. For computing the latter, we used the \texttt{dist.m} function in the Matrix Means Toolbox (MMT)\footnote{Downloaded from \textit{http://bezout.dm.unipi.it/software/mmtoolbox/}}. The second, more dramatic speedup in shown in the right panel which shows time taken to compute the matrix means
\begin{equation*}
  GM_{\ld} := \argmin_{X>0} \nlsum_{i=1}^m \ds^2(X,A_i),\quad\text{and}\quad
  GM := \argmin_{X>0} \nlsum_{i=1}^m \riem^2(X,A_i),
\end{equation*}
where $(A_1,\ldots,A_m)$ are HPD matrices. For details on $GM_{\ld}$ see Section~\ref{sec.mtxmeans}; the geometric mean $GM$ is also known as the ``Karcher mean'', and was computed using the MMT via the \texttt{rich.m} script which implements a state-of-the-art method~\cite{bruno,jeuVaVa}.

\begin{figure}[h]
  \centering
  \includegraphics[scale=0.27]{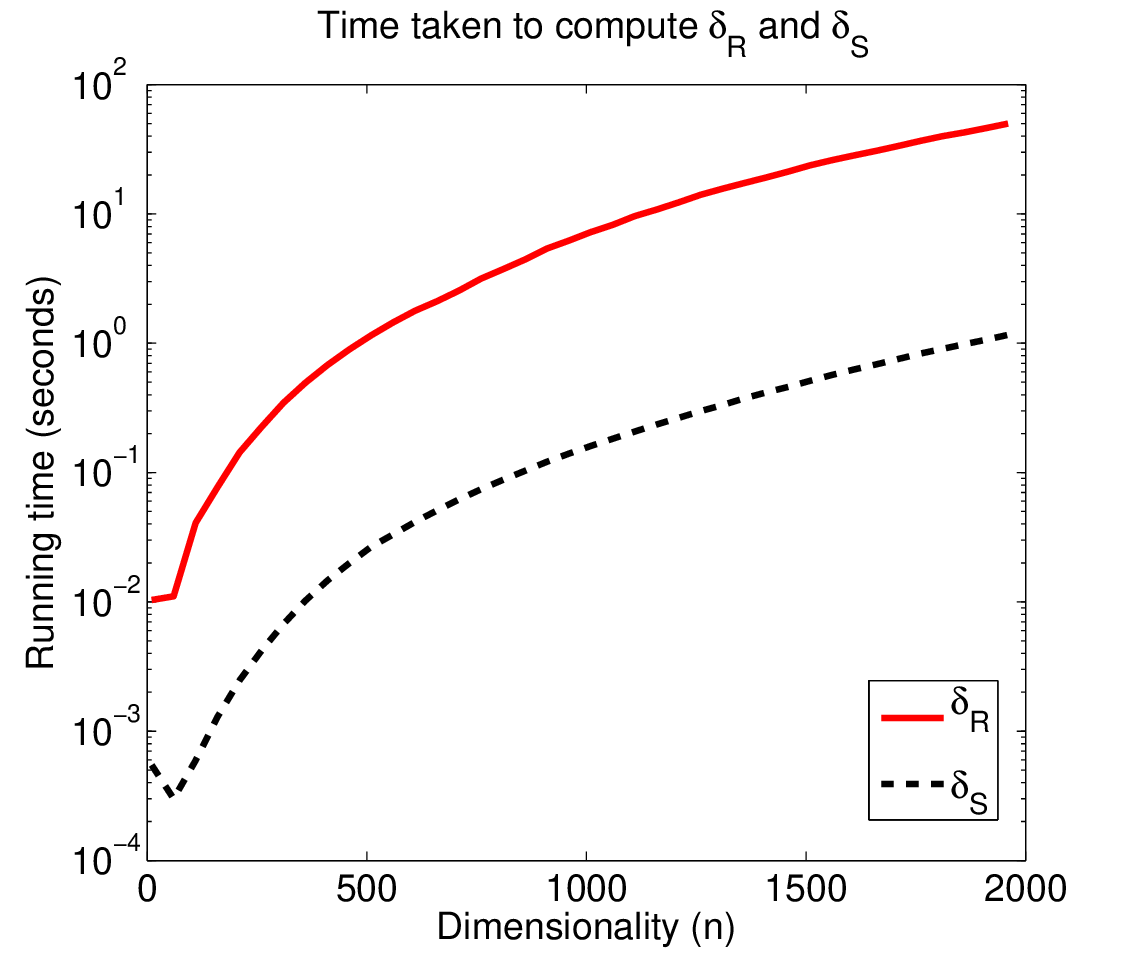} 
  \includegraphics[scale=0.27]{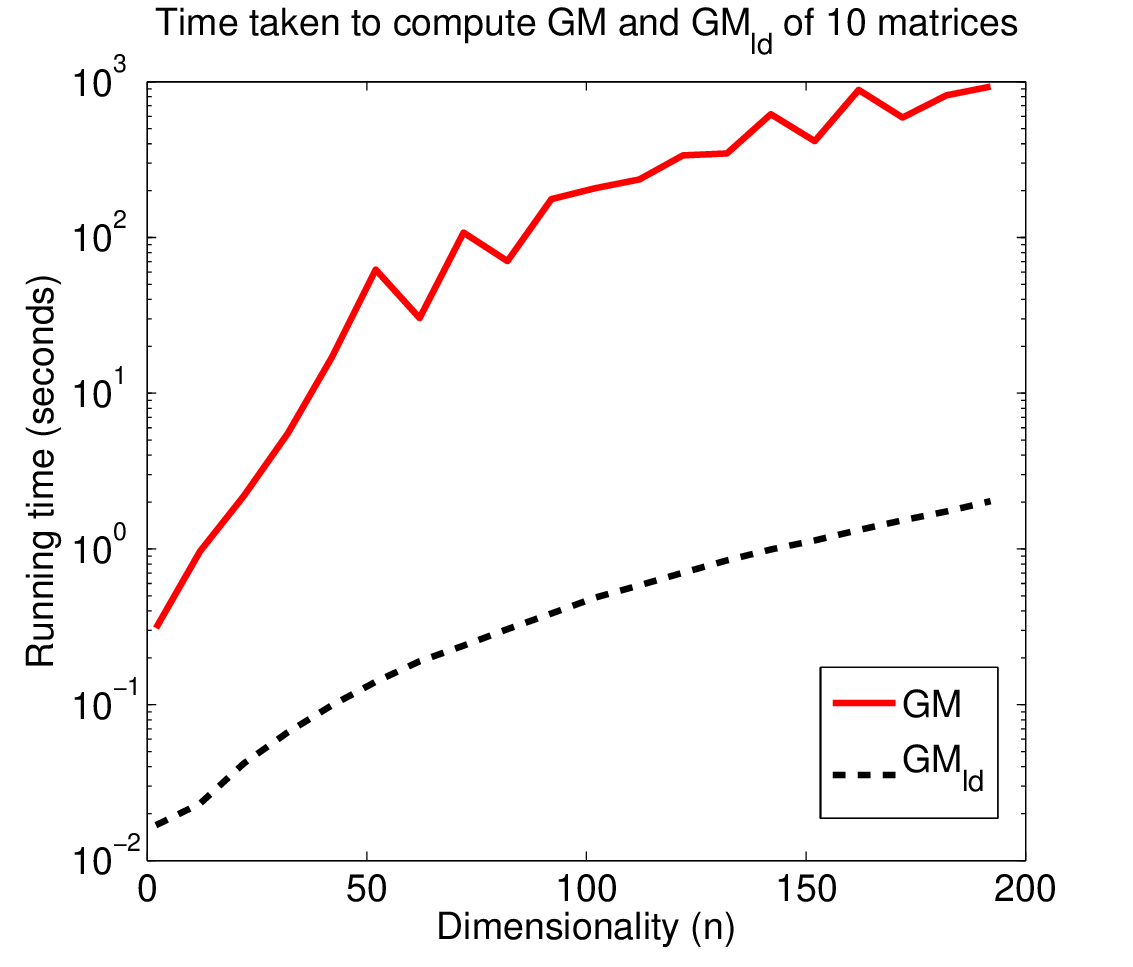}
  \caption{\textbf{Left:} Time taken to compute $\riem$ and $\ds^2$. For $\riem$, we used  MMT's Schur factorization based implementation. The results are averaged over 10 runs to reduce variance. The plot indicates that $\ds^2$ can be up to 50 times faster than $\riem$. \textbf{Right:} Time taken to compute $GM$ and $GM_{\ld}$. The former was computed using the method of~\cite{bruno}, while the latter was obtained via a fixed-point iteration. The differences are huge: $GM_{\ld}$ can be obtained up to 1000 times faster! Even the slew of methods surveyed in~\citep{jeuVaVa} show similar or worse runtimes than those cited for $\riem$ above.}
  \label{fig.one}
\end{figure}

We mention here that other alternatives to $\riem$ are also possible, for instance the popular ``log-Euclidean'' distance~\cite{lerm}, given by
\begin{equation}
  \label{eq.29}
  \delta_{le}(X,Y)=\frob{\log X - \log Y}.
\end{equation}
Notice that to compute $\delta_{le}$ we require two eigenvector decompositions; this makes it more expensive than $\ds^2$ which requires only 3 Cholesky factorizations. Even though the matrix mean under $\delta_{le}^2$ can be computed in closed form, its dependency on matrix logarithms and exponentials can make it slower than $GM_{\ld}$. However, much more importantly, for the applications in~\cite{iccv11,chSra12}, $\delta_{le}$ and other  alternatives to $\riem$ proved to be substantially less competitive than~\eqref{eq.ss}, so we limit our focus to $\ds^2$; for more extensive empirical comparisons with other distances, we refer the reader to~\cite{iccv11,chSra12}.

While our paper was under review (in 2011), we became aware of a concurrent paper of Chebbi and Moakher (CM)~\cite{chebbi}, who consider a one parameter family of divergences that generalize~\eqref{eq.ss}. Our work differs from CM in the following aspects:
\begin{list}{$\bullet$}{\leftmargin=3em}
  \setlength{\itemsep}{0pt}
\item CM prove $\ds$ to be a distance only for commuting matrices. As per Remark~\ref{rmk.commute},  the commuting case essentially reduces to the scalar case. The noncommuting case is much harder, and was conjectured to be true  in~\cite{chebbi}. We propose and solve the general noncommuting case, independent of CM.
\item We establish several theorems connecting $\ds^2$ to the Riemannian distance $\riem$. These connections have not been made either by CM or \emph{elsewhere}.
\item A question closely related to metricity of $\ds$ is whether the matrix
  \begin{equation*}
    [\det(X_i+X_j)^{-\beta}]_{i,j=1}^m,\qquad X_1,\ldots,X_m \in \pd_n,
  \end{equation*}
  is positive semidefinite for every integer $m \ge 1$ and every scalar $\beta \ge 0$. CM considered special cases of this question. We provide a complete characterization of $\beta$ necessary and sufficient for the above matrix to be semidefinite. 
\item CM analyze the ``matrix-mean'' $\min_{X>0} f(X) :=\sum_i \ds^2(X,A_i)$, whose solution they obtain by solving $\nabla f(X)=0$. CM's results essentially imply global optimality\footnote{We thank a referee for alerting us to this fact, which ultimately follows from CM's uniqueness theorem and the observation that the cost function goes to $+\infty$ for both $X \to 0$ and $X\to \infty$ (see Sec.~\ref{sec.mtxmeans} for more details).}; we provide two different proofs of this fact. One of our proofs is based of establishing \emph{geodesic convexity} of the \sdiv---a result that is also interesting because in our previous attempts~\citep{srapd,ssdiv} we oversaw this property. In fact, we show (Theorem~\ref{thm.jointgc}) that $\ds^2$ is jointly geodesically convex.
\end{list}

\vspace*{6pt}
\noindent\textbf{Other contributions.}
The present paper substantially extends our initial work~\citep{srapd}; we outline below the key differences from~\citep{srapd}.
\begin{itemize}
  \setlength{\itemsep}{0pt}
\item Due to lack of space proofs of the lemmas supporting Theorem~\ref{thm.metric} did not appear in~\citep{srapd}. In particular, proofs of Corollaries~\ref{corr.mink}, \ref{cor.eigs} and Theorems~\ref{thm.diag},  \ref{thm.det} are absent from~\citep{srapd} (these results are not difficult though).
\item None of our results on similarities between the Riemannian distance $\riem$ and $\ds$ are present in~\citep{srapd}. Although some of these results are mentioned in a summary table in~\citep{srapd}, the full theorem statements as well as their proofs are absent. The concerned results are: Theorems~\ref{thm.gm}, \ref{thm.contract}, \ref{thm.contract3}, \ref{thm.cancel}, \ref{thm.tucontract}, and~\ref{thm.contract2}. 
\item This paper uncovers a new result (previously unknown): \emph{$\ds^2$ is jointly geodesically convex}---Prop.~\ref{prop.mixedmean} and Theorem ~\ref{thm.jointgc} establish this remarkable fact.
\item This paper proves several new ``conic'' contraction results for $\ds$ and $\riem$. The appurtenant results are: Proposition~\ref{prop.logdet}, Corollary~\ref{cor.ratio}, Theorem~\ref{thm.compr}, Corollary~\ref{cor.compr.tensor}, Theorem~\ref{thm.geig}, and Corollary~\ref{corr.riem.compress}.
\item This paper proves bi-Lipschitz-like inequalities for $\riem$ and $\ds$ (Theorem~\ref{thm.bnds}).
\item Finally, this paper studies the \emph{weighted matrix-medians} problem: $$\min_{X>0}\quad \nlsum_{i=1}^mw_i\ds(X,A_i),\qquad A_i \in \pd_n,\ \text{for}\ 1 \le i \le m.$$ This problem was also partially studied by~\citet{chaCheMoa13}, who presented an iterative method that was erroneously claimed to be a fixed-point iteration in the Thompson metric. We present a counterexample to illustrate this error, and rectify it by presenting different analysis that ensures convergence (see \S\ref{sec.median}).

\end{itemize}

\section{The \sdiv} We proceed now to formally introduce the \sdiv. We follow the viewpoint of Bregman divergences. Consider, thus, a differentiable strictly convex function $f : \reals \to \reals$; then, $f(x) \ge f(y) + f'(y)(x-y)$, with equality if and only if $x=y$. The difference between the two sides of this inequality defines the \emph{Bregman Divergence}\footnote{Bregman divergences over scalars and vectors have been well-studied; see e.g.,~\cite{censor97,banerjee04b}. They are called divergences because they are not distances (though they often behave like squared distances, in a sense that can be made precise for certain choices of $f$~\citep{chen08}).} 
\begin{equation}
  \label{eq.breg.scalar}
  D_f(x,y) := f(x) - f(y) - f'(y)(x-y).
\end{equation}
The scalar divergence~\eqref{eq.breg.scalar} can be extended to Hermitian matrices. 
\begin{proposition}
  \label{prop.breg}
  Let $f$ be differentiable and strictly convex on $\reals$; let $X, Y \in \H_n$ be arbitrary. Then, we have the \emph{matrix Bregman Divergence}:
  \begin{equation}
    \label{eq.breg}
    D_f(X,Y) := \trace f(X) -\trace f(Y) - \trace\bigl(f'(Y)(X-Y)\bigr) \ge 0.
  \end{equation}
\end{proposition}
By construction $D_f$ is nonnegative, strictly convex in $X$, and zero if and only if $X=Y$. It is typically asymmetric, and may be viewed as a measure of  dissimilarity.

\begin{example}
  \label{eg.breg}
  Let $f(x) = \half x^2$. Then, for $X \in \H_n$, $\trace f(X)=\half\trace(X^2)$,  with which~\eqref{eq.breg} yields the squared \emph{Frobenius norm}
  \begin{equation*}
    D_f(X,Y) = \half\frob{X-Y}^2.
  \end{equation*}
  If $f(x)= x \log x - x$ on $(0,\infty)$, then $\trace f(X)=\half\trace(X\log X - X)$, and~\eqref{eq.breg} yields the (unnormalized) \emph{von Neumann Divergence} of quantum information theory~\cite{nielsen}:
  \begin{equation*}
    D_{\vn}(X,Y) = \trace(X\log X - X\log Y-X+Y),\qquad X, Y \in \pd_n.
  \end{equation*}
  For $f(x)=-\log x$ on $(0,\infty)$, $\trace f(X)=-\log\det(X)$, and we obtain the divergence
  \begin{equation*}
    D_{\ld}(X,Y) = \trace(\inv{Y}(X-Y)) - \log\det(X\inv{Y}),\qquad X, Y \in \pd_n,
  \end{equation*}
  which is known as the \emph{LogDet Divergence}~\cite{kulis}, or more classically as \emph{Stein's loss}~\cite{stein56}.
\end{example}

The divergence $D_{\ld}$ is of key importance to our paper, so we mention some additional noteworthy contexts where it occurs:
\begin{inparaenum}[(i)]
\item information theory~\cite{cover}, as the relative entropy between two multivariate gaussians with same mean;
\item optimization, when deriving the famous Broyden-Fletcher-Goldfarb-Shanno (BFGS) updates~\cite{nocedal};
\item matrix divergence theory \cite{bauschke,tropp,nieBha13};
\item kernel learning~\cite{kulis}.
\end{inparaenum}

Despite the broad applicability of Bregman divergences, their asymmetry is sometimes undesirable. This drawback has leads us to consider symmetric divergences, among which the most popular is the ``\emph{Jensen-Bregman}'' divergence\footnote{This symmetrization has been largely studied only for divergences over scalars or vectors.}
\begin{equation}
  \label{eq.js}
  S_f(X,Y) := D_f(X, \tfrac{X+Y}{2}) + D_f(\tfrac{X+Y}{2}, Y).
\end{equation}
This divergence has two attractive and perhaps more useful representations:
\begin{align}
  \label{eq.js1}
  S_f(X,Y) &\quad=\quad\half\bigl(\trace f(X) + \trace f(Y)\bigr) - \trace f\bigl(\tfrac{X+Y}{2}\bigr),\\
  \label{eq.js2}
  S_f(X,Y) &\quad=\quad \nlmin_Z\quad D_f(X,Z) + D_f(Y,Z).
\end{align}

Compare~\eqref{eq.breg} with~(\ref{eq.js1}): both formulas define divergence as  departure from linearity; the former uses derivatives, while the latter is stated using midpoint convexity. Representation~\eqref{eq.js1} has an advantage over~\eqref{eq.breg}, \eqref{eq.js}, and~\eqref{eq.js2}, in that it does not need to assume differentiability of $f$.

The reader must have realized by now that the \sdiv~\eqref{eq.ss} is nothing but the symmetrized divergence~\eqref{eq.js} generated by $f(x)=-\log x$. Alternatively,  the \sdiv may be essentially viewed as the Jensen-Bregman divergence between two multivariate gaussians~\cite{cover}, or as the Bhattacharya distance between them~\cite{bhatt}.

Let us now list a few basic properties of $S$.
\begin{proposition}
  \label{prop.basic}
  Let $\lambda(X)$ be the vector of eigenvalues of $X$, and $\Eig(X)$ be a diagonal  matrix with $\lambda(X)$ as its diagonal. Let $A, B, C \in \pd_n$. Then,
  \begin{enumerate}[(i)]
    \setlength{\itemsep}{1pt}
  \item $\ds(I, A) = \ds(I, \Eig(A))$;
  \item $\ds(A,B)=\ds(PAQ, PBQ)$, where $P$, $Q \in \text{GL}(n,\C)$;
  \item $\ds(A,B) =\ds(A^{-1},B^{-1})$; 
  \item $\ds^2(A\kron B, A \kron C) = n\ds^2(B,C)$; and
  \item $\ds^2(A\oplus B, C \oplus D) = \ds^2(A,C)+\ds^2(B,D)$.
  \end{enumerate}
\end{proposition}
\begin{proof}
  (i) follows from the equality $\det(I+A) = \prod_i\lambda_i(I+A) = \prod_i(1+\lambda_i(A))$.\\
  (ii) follows upon observing that
  \begin{equation*}
    \frac{\det(PAQ+PBQ)}{[\det(PAQ)]^{1/2}[\det(PBQ)]^{1/2}} =
    \frac{\det(P)\cdot\det(A+B) \cdot\det(Q)}
    {\det(P) \cdot [\det(A)]^{1/2}[\det(B)]^{1/2}\cdot\det(Q)}.
  \end{equation*}
  (iii) follows upon noting 
  \begin{equation*}
    \frac{\det(A^{-1}+B^{-1})}{[\det(A^{-1})]^{1/2}[\det(B)^{-1}]^{1/2}} = 
    \frac{\det(A)\cdot\det(A^{-1}+B^{-1}) \cdot\det(B)}
    {[\det(A)]^{1/2}[\det(B)]^{1/2}}.
  \end{equation*}
  (iv) follows as $A \kron B + A \kron C = A \kron(B+C)$, and $\det(A\kron B)=\det(A)^n\det(B)^n$.\\
  (v) is trivial since $\det(A\oplus B)=\det(A)\det(B)$.\qed
\end{proof}
The most useful corollary to Prop.~\ref{prop.basic} is \emph{congruence invariance} of $\ds$.
\begin{corollary}
  \label{corr.cong}
  Let $A, B > 0$, and let $X$ be any invertible matrix. Then,
  \begin{equation*}
    \ds(X^*AX,X^*BX) =\ds(A,B).
  \end{equation*}
\end{corollary}

The next result reaffirms that $\ds^2(\cdot,\cdot)$ is a divergence, while showing that it enjoys some limited convexity and concavity.
\begin{proposition}
  \label{prop.scvx}
  Let $A, B > 0$. Then, (i) $\ds^2(A,B) \ge 0$ with equality
  if and only if $A=B$; (ii) for fixed $B$, $\ds^2(A,B)$ is convex
  in $A$ for $A \le (1+\sqrt{2})B$, while for $A \ge (1+\sqrt{2})B$, it is concave.
\end{proposition}
\begin{proof}
  Since $\ds^2$ is a sum of Bregman divergences, property~(i) follows from
  definition~(\ref{eq.js}). Alternatively, note that $\det\left((A+B)/2\right) \ge [\det(A)]^{1/2}[\det(B)]^{1/2}$, with equality if and only if $A=B$. Part (ii) follows upon analyzing the Hessian $\nabla_A^2\ds^2(A,B)$. This Hessian can be identified with the matrix
  \begin{equation}
    \label{eq.43}
    H := \half\bigl(\inv{A}\kron\inv{A}\bigr) - \invp{A+B}\kron\invp{A+B},
  \end{equation}
  where $\kron$ is the usual the Kronecker product. Matrix $H$ is positive definite for $A \le (1+\sqrt{2})B$ and negative definite for $A \ge  (1+\sqrt{2})B$, which  proves~(ii).
\end{proof}

Below we show that $\ds^2$ is richer than a divergence: its square-root $\ds$ is actually a distance on $\pd_n$. This is the first main result of our paper. Previous authors~\cite{iccv11,chebbi} conjectured this result but could not establish it, perhaps because both ultimately sought to map $\ds$ to a Hilbert space metric. This approach fails because HPD matrices do not form even a (multiplicative) semigroup, which renders the powerful theory of harmonic analysis on semigroups~\cite{berg} inapplicable to $\ds$. This difficulty necessitates a different path to proving metricity of $\ds$, and this is the subject of the next section.

\section{The $\ssd$ metric}
\label{sec.metric}
In this section we prove the following main theorem.
\begin{theorem}
  \label{thm.metric}
  Let $\ds$ be defined by~\eqref{eq.ss}. Then, $\ssd$ is a metric on $\pd_n$.
\end{theorem}
\vskip 5pt
\noindent The proof of Theorem~\ref{thm.metric} depends on several results, which we first establish.\footnote{We note that a referee wondered whether the ``metrization'' results of \citep{chen08,chen08b} could yield an alternative proof of Thm.~\ref{thm.metric}. Unfortunately, those results rely very heavily on the commutativity and total-order available on $\reals$, both of which are missing in $\pd_n$. Indeed, for the  commutative case, Lemma~\ref{lem.cpd} alone suffices to prove that $\ds$ is a metric.}

\begin{definition}[\relax\protect{\!\!\cite[Def.~1.1]{berg}}]\normalfont
  \label{def.cnd}
  Let $\Xc$ be a nonempty set. A function $\psi: \Xc \times \Xc \to \reals$ is said to be \emph{negative definite} if for all $x, y \in \Xc$, $\psi(x,y) = \psi(y,x)$, and the inequality
  \begin{equation*}
    \nlsum_{i,j=1}^n c_ic_j\psi(x_i,x_j) \le 0,
  \end{equation*}
  holds for all integers $n \ge 2$, and subsets $\set{x_i}_{i=1}^n \subseteq \Xc$, $\set{c_i}_{i=1}^n \subseteq \reals$ with $\nlsum_{i=1}^n c_i = 0$.
\end{definition}

\begin{theorem}[\relax\protect{\!\!\cite[Prop.~3.2, Ch.~3]{berg}}]
  \label{thm.schoen}
  Let $\psi: \Xc \times \Xc \to \reals$ be negative definite. Then, there is a Hilbert space $\Hc \subseteq \reals^{\Xc}$ and a mapping $x \mapsto \varphi(x)$ from $\Xc \to \Hc$ such that one has the relation
  \begin{equation}
    \label{eq.59}
    \pnorm{\varphi(x)-\varphi(y)}{\Hc}^2 = \half(\psi(x,x) + \psi(y,y)) - \psi(x,y).
  \end{equation}
  Moreover, negative definiteness of $\psi$ is necessary for such a mapping to exist.
\end{theorem}

Theorem~\ref{thm.schoen} helps prove the triangle inequality for the scalar case.
\begin{lemma}
  \label{lem.cpd}
  Define, the scalar version of $\sqrt{\jsld}$ as
  \begin{equation*}
    \smet(x,y) := \sqrt{\log[(x+y)/(2\sqrt{xy})]},\quad x, y > 0.
  \end{equation*}
  Then, $\smet$ satisfies the triangle inequality, i.e.,
  \begin{equation}
    \label{eq.11}
    \smet(x,y) \le \smet(x,z) + \smet(y,z)\quad\text{for all}\ \ x, y, z > 0.
  \end{equation}
\end{lemma}
\begin{proof}
  We show that $\psi(x,y) = \log((x+y)/2)$ is negative definite. Since
  $\smet^2(x,y) = \psi(x,y)-\half(\psi(x,x)+\psi(y,y))$, 
  Theorem~\ref{thm.schoen} then immediately implies the triangle
  inequality~\eqref{eq.11}. To prove that $\psi$ is negative definite, by~\cite[Thm.~2.2, Ch.~3]{berg} we may equivalently show that $e^{-\beta\psi(x,y)} = \pfrac{x+y}{2}^{-\beta}$ is a positive
  definite function for $\beta > 0$, and $x, y > 0$. To that end, it suffices to show that the matrix
  \begin{equation*}
    H = [h_{ij}] = \left[(x_i+x_j)^{-\beta}\right],\quad 1 \le i,j \le n,
  \end{equation*}
  is HPD for every integer $n \ge 1$, and positive numbers $\set{x_i}_{i=1}^n$. Now, observe that 
  \begin{equation}
    \label{eq.12}
    h_{ij} = \frac{1}{(x_i+x_j)^\beta} = \frac{1}{\Gamma(\beta)}\int_0^\infty e^{-t(x_i+x_j)} t^{\beta-1}dt,
  \end{equation}
  where $\Gamma(\beta) = \int_0^\infty e^{-t}t^{\beta-1}dt$ is the Gamma function. Thus, with $f_i(t) = e^{-tx_i}t^{\frac{\beta-1}{2}} \in L_2([0,\infty))$, we see that $[h_{ij}]$ equals the Gram matrix $[\ip{f_i}{f_j}]$, whereby $H \ge 0$.\qed
\end{proof}
Using Lemma~\ref{lem.cpd} obtain the following ``Minkowski'' inequality for $\smet$.
\begin{corollary}
  \label{corr.mink}
  Let $x, y, z \in \reals^n_{++}$; and let $p \ge 1$. Then,
  \begin{equation}
    \label{eq.36}
    \left(\nlsum_i \smet^p(x_i,y_i)\right)^{1/p} \le     \left(\nlsum_i \smet^p(x_i,z_i)\right)^{1/p} + 
    \left(\nlsum_i \smet^p(y_i,z_i)\right)^{1/p}.
  \end{equation}
\end{corollary}
\begin{proof}
  Lemma~\ref{lem.cpd} implies that for positive scalars $x_i$, $y_i$, and $z_i$, we have
  \begin{equation*}
    \smet(x_i,y_i) \le \smet(x_i,z_i) + \smet(y_i,z_i),\qquad 1 \le i \le n.
  \end{equation*}%
  Exponentiate, sum, and invoke Minkowski's inequality to conclude~\eqref{eq.36}.\qed
\end{proof}

\begin{theorem}
  \label{thm.diag}
  Let $X, Y, Z > 0$ be diagonal matrices. Then,
  \begin{equation}
    \label{eq.37}
    \ssd(X,Y) \le \ssd(X, Z) + \ssd(Y, Z)
  \end{equation}
\end{theorem}
\begin{proof}
  For diagonal matrices $X$ and $Y$, it is easy to verify that $\ssd^2(X,Y) = \nlsum_i \smet^2(X_{ii}, Y_{ii}).$ Now invoke Corollary~\ref{corr.mink} with $p=2$.\qed
\end{proof}

Next, we recall an important determinantal inequality for positive matrices.
\begin{theorem}[\protect{\!\!\cite[Exercise~VI.7.2]{bhatia97}}]
  \label{thm.det}
  Let $A, B > 0$. Let $\lambda^\downarrow(X)$ denote the vector of eigenvalues of $X$ sorted in decreasing order; define $\lambda^\uparrow(X)$ likewise. Then,
  \begin{equation}
    \label{eq.14}
    \nlprod_{i=1}^n(\lambda_i^\downarrow(A) + \lambda_i^\downarrow(B)) \le \det(A+B) \le
    \nlprod_{i=1}^n(\lambda_i^\downarrow(A)+\lambda_i^\uparrow(B)).
  \end{equation}
\end{theorem}
\begin{corollary}
  \label{cor.eigs}
  Let $A, B > 0$. Let $\Eig^\downarrow(X)$ denote the diagonal matrix with $\lambda^\downarrow(X)$ as its diagonal; define $\Eig^\uparrow(X)$ likewise. Then,
  \begin{alignat*}{3}
    \ssd(\Eig^\downarrow(A), \Eig^\downarrow(B))  &\le\ \ssd(A,B)\ \ &\le&\ \ \ssd(\Eig^\downarrow(A), \Eig^\uparrow(B)).
  \end{alignat*}
\end{corollary}
\begin{proof}
  Scale $A$ and $B$ by 2, divide each term in~\eqref{eq.14} by $\sqrt{\det(A)\det(B)}$, and note that $\det(X)$ is invariant to permutations of $\lambda(X)$; take logarithms to conclude.\qed
\end{proof}

The final result we need is well-known in linear algebra (we provide a proof).
\begin{lemma}
  \label{lem.diag2}
  Let $A > 0$, and let $B$ be Hermitian. There is a matrix $P$ for which
  \begin{equation}
    \label{eq.32}
    P^*AP = I,\quad\text{and}\quad P^*BP = D,\quad{where}\ D\ \text{is diagonal}.
  \end{equation}
\end{lemma}
\begin{proof}
  Let $A=U\Lambda U^*$, and define $S=\Lambda^{-1/2}U$. The
  the matrix $S^*U^*BSU$ is Hermitian; so let $V$ diagonalize it to $D$. Set $P=USV$, to obtain
  \begin{equation*}
    P^*AP = V^*S^*U^*U\Lambda U^*USV = V^*U^*\Lambda^{-1/2}\Lambda \Lambda^{-1/2}UV = I;
  \end{equation*}
  and by construction, it follows that $P^*BP = V^*S^*U^*BUSV=D$.\qed
\end{proof}

Accoutered with the above results, we can finally prove Theorem~\ref{thm.metric}.

\begin{proof}{(Theorem~\ref{thm.metric}).\hskip4pt}
  We need to show that $\ssd$ is symmetric, nonnegative, definite, and that is satisfies the triangle inequality. Symmetry is obvious. Nonnegativity and definiteness were shown in Prop.~\ref{prop.scvx}. The only hard part is to prove the triangle inequality, a result that has eluded previous attempts \citep{iccv11,chebbi}.

  Let $X, Y, Z > 0$ be arbitrary. From Lemma~\ref{lem.diag2} we know that there is a matrix $P$ such that $P^*XP=I$ and $P^*YP=D$. Since $Z > 0$ is arbitrary, and congruence preserves positive definiteness, we may write just $Z$ instead of $P^*ZP$. Also, since $\ssd(P^*XP,P^*YP)=\ssd(X,Y)$ (see Prop.~\ref{prop.basic}), proving the triangle inequality reduces to showing that
  \begin{equation}
    \label{eq.35}
    \ssd(I,D) \le \ssd(I, Z) + \ssd(D, Z).
  \end{equation}
  Consider now the diagonal matrices $D^\downarrow$ and $\Eig^\downarrow(Z)$. Theorem~\ref{thm.diag} asserts
  \begin{equation}
    \label{eq.26}
    \ssd(I, D^\downarrow) \le \ssd(I, \Eig^{\downarrow}(Z)) + \ssd(D^\downarrow, \Eig^\downarrow(Z)).
  \end{equation}
  Prop.~\ref{prop.basic}(i) implies that $\ssd(I,D)=\ssd(I,D^\downarrow)$ and 
  $\ssd(I,Z)=\ssd(I,\Eig^\downarrow(Z))$, while Corollary~\ref{cor.eigs} shows that $\ssd(D^\downarrow, \Eig^\downarrow(Z)) \le \ssd(D,Z)$. Combining these inequalities, we immediately obtain~\eqref{eq.35}.
\end{proof}

We now turn our attention to a connection of importance to machine learning and approximation theory: kernel functions related to $\ssd$. Indeed, some of connections (e.g., Theorem~\ref{thm.wallach}) have already been recently useful in computer vision~\citep{harandi12}.


\subsection{Hilbert space embedding}
\label{sec.hilb}
Since $\ssd$ is a metric, and Lemma~\ref{lem.cpd} shows that for scalars, $\ssd$ embeds isometrically into a Hilbert space, one may ask if $\ssd(X,Y)$ also admits such an embedding. But as mentioned previously, it is the lack of such an embedding that necessitated a different route to metricity. Let us look more carefully at what goes wrong, and what kind of Hilbert space embeddability does $\ssd^2$ admit.

Theorem~\ref{thm.schoen} implies that a Hilbert space embedding exists if and only if $\ssd^2(X,Y)$ is a negative definite kernel; equivalently, if and only if the map (\emph{cf.} Lemma~\ref{lem.cpd})
\begin{equation*}
  e^{-\beta\ssd^2(X,Y)} = \frac{\det(X)^\beta\det(Y)^\beta}{\det((X+Y)/2)^\beta},
\end{equation*}
is a positive definite kernel for $\beta > 0$. It suffices to check whether the matrix
\begin{equation}
  \label{eq.60}
  H_\beta = [h_{ij}] = \left[\det(X_i+X_j)^{-\beta} \right],\quad 1 \le i, j \le m,
\end{equation}
is positive definite for every $m \ge 1$ and arbitrary HPD matrices $X_1,\ldots, X_m \in \pd_n$. 

Unfortunately, a quick numerical experiment reveals that $H_\beta$ can be indefinite. A counterexample is given by the following positive definite matrices ($m=5$, $n=2$)
\begin{equation}
  \label{eq.61}
  \begin{split}
    X_1 =
    &\begin{bmatrix}
      0.1406 &0.0347\\
      0.0347 &0.1779
    \end{bmatrix},\ 
    X_2 =
    \begin{bmatrix}
      2.0195 &0.0066\\
      0.0066 &0.2321
    \end{bmatrix},\
    X_3 =
    \begin{bmatrix}
      1.0924&0.0609\\
      0.0609&1.2520
    \end{bmatrix},\\
    &X_4 =
    \begin{bmatrix}
      1.0309&0.8694\\
      0.8694&1.2310
    \end{bmatrix},\ \text{and}\ \ 
    X_5 =
    \begin{bmatrix}
      0.2870&-0.4758\\
      -0.4758&2.3569
    \end{bmatrix},
  \end{split}
\end{equation}
and by setting $\beta=0.1$, for which $\lambda_{\min}(H_\beta) = -0.0017$. This counterexample destroys hopes of embedding the metric space $(\pd_n,\ssd)$ isometrically into a Hilbert space.

Although matrix~\eqref{eq.60} is not HPD in general, we might ask: \emph{For what choices of $\beta$ is $H_\beta$ HPD?} Theorem~\ref{thm.wallach} answers this question for $H_\beta$ formed from symmetric real positive definite matrices, and characterizes the values of $\beta$ necessary and sufficient for $H_\beta$ to be positive definite. We note here that the case $\beta=1$ was essentially treated in~\citep{cuturi}, in the context of semigroup kernels on measures.

\begin{theorem}
  \label{thm.wallach}
  Let $X_1,\ldots,X_m$ be real symmetric matrices in $\pd_n$. The $m \times m$ matrix $H_\beta$ defined by~\eqref{eq.60} is positive definite, if and only if $\beta$ satisfies
  \begin{equation}
    \label{eq.62}
    \beta \in \set{\tfrac{j}{2}: j \in \mathbb{N},\ \text{and}\ 1 \le j \le (n-1)} \cup \set{\gamma : \gamma \in \reals, \text{and}\ \gamma > \half(n-1)}.
  \end{equation} 
\end{theorem}
\begin{proof}
  We first prove the ``if'' part. Recall therefore, the Gaussian integral
  \begin{equation*}
    \int_{\reals^n} e^{-x^T X x}dx\quad=\quad\pi^{n/2}\det(X)^{-1/2}.
  \end{equation*}
  Define the map $f_i := \frac{1}{\pi^{n/4}}e^{-x^T X_i x} \in L_2(\reals^n)$, where the inner-product is given by
  \begin{equation*}
    \ip{f_i}{f_j} := \frac{1}{\pi^{n/2}}\int_{\reals^n} e^{-x^T(X_i+X_j)x}dx = \det(X_i+X_j)^{-1/2}.
  \end{equation*}
  Thus, it follows that $H_{1/2} \ge 0$. The Schur product theorem says that the elementwise product of two positive matrices is again positive. So, in particular $H_\beta$ is positive whenever $\beta$ is an integer multiple of $1/2$. To extend the result to all $\beta$ covered
  by~\eqref{eq.62}, we invoke another integral representation: the
  \emph{multivariate Gamma function}, defined as~\cite[\S2.1.2]{muirhead}
  \begin{equation}
    \label{eq.48}
    \Gamma_n(\beta) := \int_{\pd_n} e^{-\trace(A)}\det(A)^{\beta-(n+1)/2}dA,\qquad\text{where}\ \beta > \half(n-1).
  \end{equation}
  Let $f_i := ce^{-\trace(AX_i)} \in L_2(\pd_n)$, for some constant $c$; then, compute the inner product
  \begin{equation*}
    \ip{f_i}{f_j} := c'\int_{\pd_n}e^{-\trace(A(X_i+X_j))}\det(A)^{\beta-(n+1)/2}dA = \det(X_i+X_j)^{-\beta},
  \end{equation*}
  which exists if $\beta > \half(n-1)$. Thus, $H_\beta \ge 0$ for all $\beta$ defined by~\eqref{eq.62}.

  The converse is a deeper result grounded in the theory of symmetric cones. Specifically, since $\pd_n$ is a symmetric cone, and $1/\det(X)$ is decreasing on this cone, an appeal to \citep[Thm.~VII.3.1]{koryan} yields the only if part of our claim.
\end{proof}



\begin{remark}
  \label{rmk.commute}
  Let $\Xc$ be a set of HPD matrices that commute with each other. Then, $(\Xc,\ssd)$ can be isometrically embedded into a Hilbert space. This claim follows because a commuting set of matrices can be simultaneously diagonalized, and for diagonal matrices, $\ssd^2(X,Y) = \sum_i \smet^2(X_{ii},Y_{ii})$, which is a nonnegative sum of negative definite kernels and is therefore itself negative definite.
\end{remark}

Theorem~\ref{thm.wallach} shows that $e^{-\beta\ssd^2}$ is not a kernel for all $\beta > 0$, while Remark~\ref{rmk.commute} mentions an extreme case for which $e^{-\beta\ssd^2}$ is always a positive definte kernel. This prompts us to pose the following problem.

\vspace*{8pt}
\begin{center}
  \begin{minipage}{.9\linewidth}
    \begin{minipage}{1.0\linewidth}
      \textbf{Open problem 1.} Determine necessary and sufficient conditions on a set $\Xc \subset
      \pd_n$, so that $e^{-\beta\ssd^2(X,Y)}$ is a kernel function on
      $\Xc \times \Xc$ for all $\beta > 0$.
    \end{minipage}
  \end{minipage}
\end{center}
\vspace*{8pt}

\section{Connections with $\riem$}
\label{sec.riem}
This section returns to our original motivation: using $\jsld$ as an alternative to the Riemannian metric $\riem$. In particular, in this section we show a sequence of results that highlight similarities between $\jsld$ and $\riem$. Table~\ref{tab.summ} lists these to provide a quick summary. Thereafter, we develop the details.

\begin{table}[h]\small
  \hskip-12pt
  \begin{tabular}{l|@{\hspace*{3pt}}l||l|@{\hspace*{3pt}}l}
    Riemannian distance & Ref. & \sdiv & Ref.\\
    \hline
    $\riem(X^*AX,X^*BX)= \riem(A,B)$          &\citep[Ch.6]{bhatia07}     & $\ds(X^*AX,X^*BX)= \ds(A,B)$          & Prp.\ref{prop.basic}\\
    $\riem(\inv{A},\inv{B}) = \riem(A,B)$     &\citep[Ch.6]{bhatia07}     & $\ds(\inv{A},\inv{B}) = \ds(A,B)$     & Prp.\ref{prop.basic}\\
    $\riem(A\kron B, A\kron C) = n\riem(B,C)$ & E                          & $\ds^2(A\kron B, A\kron C) = n\ds^2(B,C)$ & Prp.\ref{prop.basic}\\[4pt]
    \hline
    $\riem(A^t,B^t) \le t\riem(A,B)$           & \citep[Ex.6.5.4]{bhatia07}  & $\ds^2(A^t,B^t)  \le t\ds^2(A,B)$         & Th.\ref{thm.contract}\\
    $\riem(A^s,B^s) \le (s/u)\riem(A^u,B^u)$   & Th.\ref{thm.tucontract}   & $\ds^2(A^s,B^s) \le (s/u)\ds^2(A^u,B^u)$  & Th.\ref{thm.tucontract}\\
    $\riem(X,A)$ g-convex in $X$               & \citep[6.1.11]{bhatia07}          & $\ds^2(X,Y)$ g-convex in $X, Y$ & Th.\ref{thm.jointgc}\\
    $\riem(X^*AX,X^*BX) \le\riem(A,B)$         & Cor.~\ref{corr.riem.compress} & $\ds(X^*AX,X^*BX) \le \ds(X,Y)$       & Th.\ref{thm.compr}\\[4pt]
    \hline
    $\riem(A,A\gm B) = \riem(B,A\gm B)$        & E                          & $\ds(A,A\gm B) = \ds(B,A\gm B)$       & Th.\ref{thm.gm}\\
    $\riem(A,A\gm_t B) = t\riem(A,B)$          &\citep[Th.6.1.6]{bhatia07} & $\ds^2(A,A\gm_t B) \le t\ds^2(A,B)$       & Th.\ref{thm.contract3}\\
    $\riem(A\gm_t B, A\gm_t C) \le t \riem(B,C)$ & \citep[Th.6.1.2]{bhatia07}& $\ds^2(A\gm_t B, A\gm_t C) \le t\ds^2(B,C)$ & Th.\ref{thm.cancel}\\[1pt]
    $\min_X\riem^2(X,A)+\riem^2(X,B)$         & \citep[\S6.2.8]{bhatia07}    & $\min_X\ssd^2(X,A)+\ssd^2(X,B)$          & Th.\ref{thm.gm}\\[4pt]
    \hline
    $\riem(A+X,A+Y) \le \riem(X,Y)$           & \citep{bougerol}            & $\ds^2(A+X, A+Y) \le \ds^2(X,Y)$        & Th.\ref{thm.contract2}
  \end{tabular}
  \caption{\small Similarities between $\riem$ and $\ds$,$\ds^2$ at a glance. All matrices are assumed to be in $\pd_n$, except for $X$ in Line 1, $X \in \text{GL}_n(\C)$, and in Line 7, $X \in \C^{n\times k}$ ($k \le n$, full colrank). The scalars $t, s, u$ satisfy $0 < t \le 1$, $1\le s\le u <\infty$. An `E' indicates an easily verifiable result.}\label{tab.summ}
\end{table}

\subsection{Geometric mean}
\label{sec.gm}
We begin by studying an object that connects $\riem$ and $\jsld$ most intimately: the matrix \emph{geometric mean} (GM). For HPD matrices $A$ and $B$, the GM is denoted by $A \gm B$, and is given by the formula
\begin{equation}
  \label{eq.47}
  A \gm B := A^{1/2}(A^{-1/2}BA^{-1/2})^{1/2}A^{1/2}.
\end{equation}
The GM~\eqref{eq.47} has numerous attractive properties---see for instance~\cite{ando79}---among which, the following variational characterization is very important~\cite{bhatiaholb}:
\begin{equation}
  \label{eq.123}
  \begin{split}
    A \gm B = &\argmin\nolimits_{X > 0}\quad \riem^2(A, X) + \riem^2(B, X),\quad\text{and}\\
    &\riem(A, A \gm B) = \riem(B, A \gm B).
  \end{split}
\end{equation}
Surprisingly, the GM enjoys a similar characterization even with $\ssd^2$.
\begin{theorem}
  \label{thm.gm}
  Let $A, B > 0$. Then, 
  \begin{equation}
    \label{eq.111}
    A \gm B = \argmin\nolimits_{X > 0}\quad \left[h(X) := \ssd^2(X, A) + \ssd^2(X, B)\right].
  \end{equation}
  Moreover, $A \gm B$ is equidistant from $A$ and $B$, i.e., $\ssd(A, A \gm B) = \ssd(B, A \gm B)$.
\end{theorem}
\begin{proof}
  If $A=B$, then clearly $X=A$ minimizes $h(X)$. Assume therefore, that $A \neq B$. Ignoring the constraint $X > 0$ for the moment, we see that any stationary point of $h(X)$ must satisfy $\nabla h(X) = 0$. This condition translates into
  \begin{align}
    \nonumber
    \nabla h(X) = &\left(\tfrac{X+A}{2}\right)^{-1}\tfrac{1}{2} + \left(\tfrac{X+B}{2}\right)^{-1}\tfrac{1}{2} - X^{-1} = 0,\\
    \label{eq.33}
    &\implies\qquad X^{-1} = (X+A)^{-1}+(X+B)^{-1}\\
    \nonumber
    &\implies\qquad (X+A)X^{-1}(X+B) = 2X + A + B\\
    \nonumber
    &\implies\qquad B = XA^{-1}X.
  \end{align}
  The last equation is a Riccati equation whose \emph{unique, positive definite} solution is the geometric mean $X=A \gm B$ (see~\cite[Prop~1.2.13]{bhatia07}).

  Next, we show that this stationary point is a local minimum, not a local
  maximum or saddle point. To that end, we show that the Hessian is positive definite at the stationary point $X=A\gm B$. The Hessian of $h(X)$ is given by
  \begin{equation*}
    2\nabla^2h(X) = X^{-1} \kron X^{-1} - \left[(X+A)^{-1} \kron (X+A)^{-1} + (X+B)^{-1}\kron (X+B)^{-1} \right].
  \end{equation*}
  Writing $P=(X+A)^{-1}$, and $Q=(X+B)^{-1}$, upon using~\eqref{eq.33} we obtain
  \begin{equation*}
    \begin{split}
      2\nabla^2h(X) &= (P+Q) \kron (P+Q) - P \kron P - Q \kron Q\\
      &= (P+Q)\kron P + (P+Q)\kron Q - P \kron P - Q \kron Q\\
      &= (Q\kron P) + (P \kron Q) > 0.
    \end{split}
  \end{equation*}
  Thus, $X = A \gm B$ is a \emph{strict} local minimum of~\eqref{eq.11}. This local minimum is actually global as $\nabla h(X)=0$ has a unique positive solution and $h$ goes to $+\infty$ at the boundary.

  To prove the equidistance, recall that $A \gm B = B \gm A$; then observe that
  \begin{align*}
    \jsld(A, A \gm B) &= \jsld(A, B \gm A) = 
    \jsld(A, B^{1/2}(B^{-1/2}AB^{-1/2})^{1/2}B^{1/2})\\
    &= \jsld(B^{-1/2}AB^{-1/2}, (B^{-1/2}AB^{-1/2})^{1/2})\\
    &= \jsld(B^{1/2}(B^{-1/2}AB^{-1/2})^{1/2}B^{1/2}, B)\\ 
    &= \jsld(B \gm A, B) = \jsld(B, A \gm B).\hskip2in\qed
  \end{align*}
\end{proof}

\subsection{Geodesic convexity}
The above derivation concludes optimality from first principles. It was originally driven by the fact that $\ssd$ is not geodesically convex~\citep{srapd}. However, we recently realized that the square $\ssd^2$ \emph{is} actually geodesically convex. This realization leads to more insightful proof of uniqueness and optimality of the S-mean. 

In fact even more is true: Theorem~\ref{thm.jointgc} shows that $\ssd^2(X,Y)$ is not only geodesically convex, it is \emph{jointly} geodesically convex. 
Before proving Theorem~\ref{thm.jointgc}, we recall two useful results; the first immediately implies the second\footnote{It is a minor curiosity to note that \citet[Thm.~2]{mond96} proved a mixed-mean inequality for the matrix geometric and arithmetic means; Prop.~\ref{prop.mixedmean} includes their result as a special case.}.

\begin{theorem}[\!\!\protect{\cite{kuboAndo80}}]
  \label{thm.andomax}
  The GM of $A, B \in \pd_n$ is given by the variational formula
  \begin{equation*}
    \label{eq.31}
    A\gm B = \max\bigl\lbrace X \in \H_n \mid
    \mbox{\scriptsize $\begin{bmatrix}
        A & X\\
        X & B
      \end{bmatrix}$} \ge 0\bigr\rbrace.
  \end{equation*}
\end{theorem}

\begin{proposition}[Joint-concavity (see e.g.~\cite{kuboAndo80})]
  \label{prop.mixedmean}
  Let $A,B,C,D > 0$. Then,
  \begin{equation}
    \label{eq.30}
    (A\gm B)+(C\gm D) \le (A+C)\gm(B+D).
  \end{equation}
\end{proposition}
\begin{proof}
  From one application of Thm.~\ref{thm.andomax} we have
  \begin{equation*}
    0 \preceq \begin{bmatrix}
      A & A\gm B\\
      A\gm B & B
    \end{bmatrix} +
    \begin{bmatrix}
      C & C\gm D\\
      C\gm D & D
    \end{bmatrix} =
    \begin{bmatrix}
      A+C & A\gm B + C\gm D\\
      A\gm B + C\gm D & B+D
    \end{bmatrix}.
  \end{equation*}
  A second application of Thm.~\ref{thm.andomax} then immediately yields~\eqref{eq.30}.
\end{proof}

\begin{theorem}
  \label{thm.jointgc}
  The function $\ssd^2(X,Y)$ is jointly g-convex for $X, Y > 0$.
\end{theorem}
\begin{proof}
  Since $\ssd^2$ is continuous, it suffices to show that for $X_1,X_2,Y_1,Y_2 > 0$,
  \begin{equation}
    \label{eq.sgc}
    \ssd^2(X_1\gm X_2, Y_1\gm Y_2) \le \half\ssd^2(X_1,Y_1) + \half\ssd^2(X_2,Y_2).
  \end{equation}
  From Prop.~\ref{prop.mixedmean} it follows that
  \begin{equation*}
    X_1\gm X_2 + Y_1\gm Y_2 \le (X_1+Y_1)\gm(X_2+Y_2).
  \end{equation*}
  Since $\log\det$ is monotonic and determinant is multiplicative, this inequality implies
  \begin{equation*}
    \begin{split}
      \log\det\left(\tfrac{X_1\gm X_2 + Y_1\gm Y_2}{2}\right) &\le \log\det\left(\tfrac{(X_1+Y_1)\gm(X_2+Y_2)}{2}\right)\\
      &=\half\log\det\left(\tfrac{X_1+Y_1}{2}\right) + \half\log\det\left(\tfrac{X_2+Y_2}{2}\right).
    \end{split}
  \end{equation*}
  Combining this inequality with the identity
  \begin{equation}
    \label{eq.38}
    -\half\log\det\bigl((X_1\gm X_2)(Y_1\gm Y_2)\bigr) = -\tfrac{1}{4}\log\det(X_1Y_1) - \tfrac{1}{4}\log\det(X_2Y_2)
  \end{equation}
  we obtain inequality~\eqref{eq.sgc}, establishing the joint g-convexity.\qed
\end{proof}

Since $\ds^2$ is geodesically convex (to be precise, we define $\ds$ to be $+\infty$ whenever either of its arguments fails to be strictly positive), the objective function $h(X) = \nlsum_i w_i\ds^2(X,A_i)$ is also geodesically convex. A quick observation shows that $h(0)=h(\infty)=+\infty$, which suggests that $h$ attains its minimum. Since it is strictly geodesically convex, the solution to $\nabla h(X)=0$ is unique, and yields the desired minimum.

\subsection{Basic contraction results}
\label{sec.cont1}
In this section we show that $\ds$ and $\riem$ share several contraction properties. We will state properties either in terms of $\ds^2$ or $\ds$, depending on whichever appears more elegant.

\subsubsection{Power-contraction}
The metric $\riem$ satisfies (e.g.,~\cite[Exercise~6.5.4]{bhatia07})
\begin{equation}
  \label{eq.44}
  \riem(A^t,B^t) \le t\riem(A,B),\quad\text{for}\quad A, B > 0\ \text{and}\ t \in [0,1].
\end{equation}
Theorem~\ref{thm.contract} shows that \sdiv satisfies the same relation.
\begin{theorem}
  \label{thm.contract}
  Let $A, B > 0$, and let $t \in [0,1]$. Then,
  \begin{equation}
    \label{eq.45}
    \ds^2(A^t,B^t) \le t\ds^2(A,B). 
  \end{equation}
  Moreover, if $t\ge 1$, then the inequality gets reversed.
\end{theorem}
\begin{proof}
  Recall that for $t \in [0,1]$, the map $X \mapsto X^t$ is operator concave. Thus, $\half(A^t+B^t) \le \left(\frac{A+B}{2}\right)^t$; by monotonicity of the determinant it then follows that
  \begin{equation*}
    \ds^2(A^t,B^t) = \log \frac{\det\left(\half(A^t+B^t)\right)}{\det(A^tB^t)^{1/2}} \le 
    \log \frac{\det\left(\half(A+B)\right)^t}{\det(AB)^{t/2}} = t\ds^2(A,B).
  \end{equation*}
  The reverse inequality for $t \ge 1$,
  follows from~\eqref{eq.45} by considering $\ds^2(A^{1/t},B^{1/t})$.\qed
\end{proof}

\subsubsection{Contraction on geodesics}
\label{sec.geo}
The curve
\begin{equation}
  \label{eq.46}
  \gamma(t) := A^{1/2}(A^{-1/2}BA^{-1/2})^tA^{1/2},\quad\text{for}\ t \in [0,1],
\end{equation}
parameterizes the \emph{unique} geodesic between the positive matrices $A$ and $B$ on the manifold $(\pd_n,\riem)$~\cite[Thm.~6.1.6]{bhatia07}. On this curve $\riem$ satisfies 
\begin{equation*}
  \riem(A, \gamma(t)) = t\riem(A,B),\quad t \in [0,1].
\end{equation*}
The \sdiv satisfies a similar, albeit slightly weaker result.
\begin{theorem}
  \label{thm.contract3}
  Let $A, B > 0$, and $\gamma(t)$ be defined by~\eqref{eq.46}. Then,
  \begin{equation}
    \label{eq.10}
    \ds^2(A, \gamma(t)) \le t\ds^2(A,B),\qquad 0 \le t \le 1.
  \end{equation}
\end{theorem}
\begin{proof} The proof follows upon observing that
  \begin{align*}
    \ds^2(A,\gamma(t)) &=\ds^2(I, (A^{-1/2}BA^{-1/2})^t)\\ 
    &\hskip-2pt\stackrel{~(\ref{eq.45})}{\le} t\ds^2(I, A^{-1/2}BA^{-1/2}) =t\ds^2(A,B).~\hskip1cm\qed
  \end{align*}
\end{proof}

\begin{figure}[t]
  \centering
  \includegraphics[width=.33\linewidth]{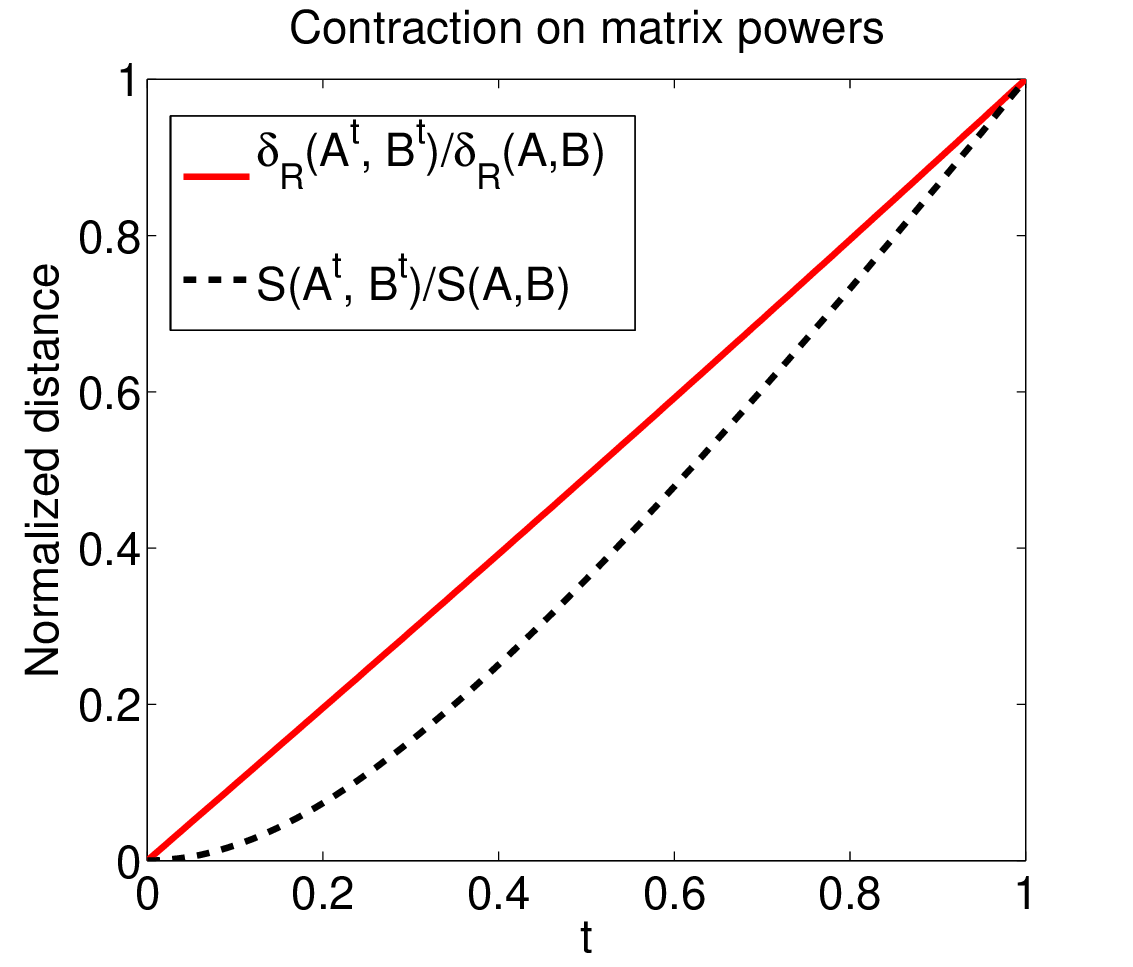}\hskip-8pt
  \includegraphics[width=.33\linewidth]{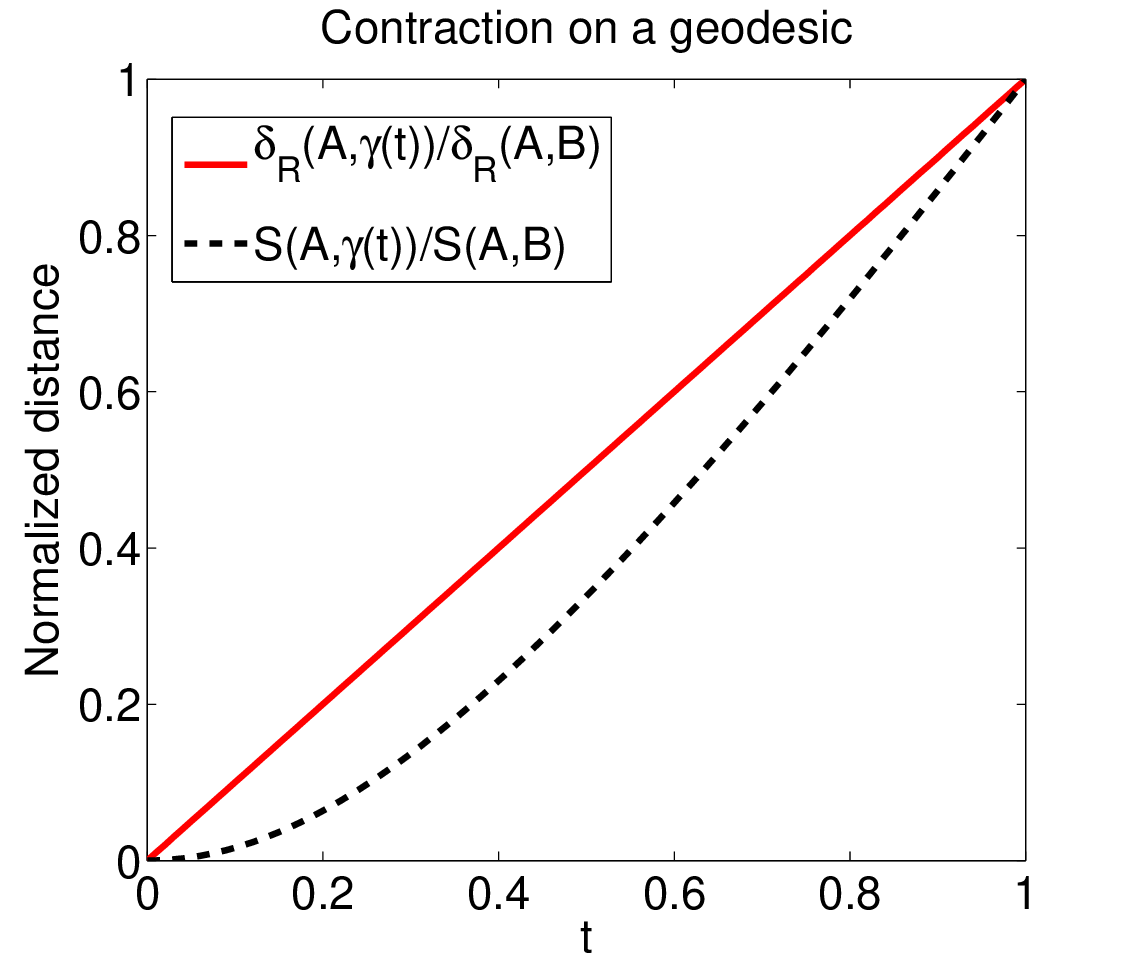}\hskip-8pt \includegraphics[width=.33\linewidth]{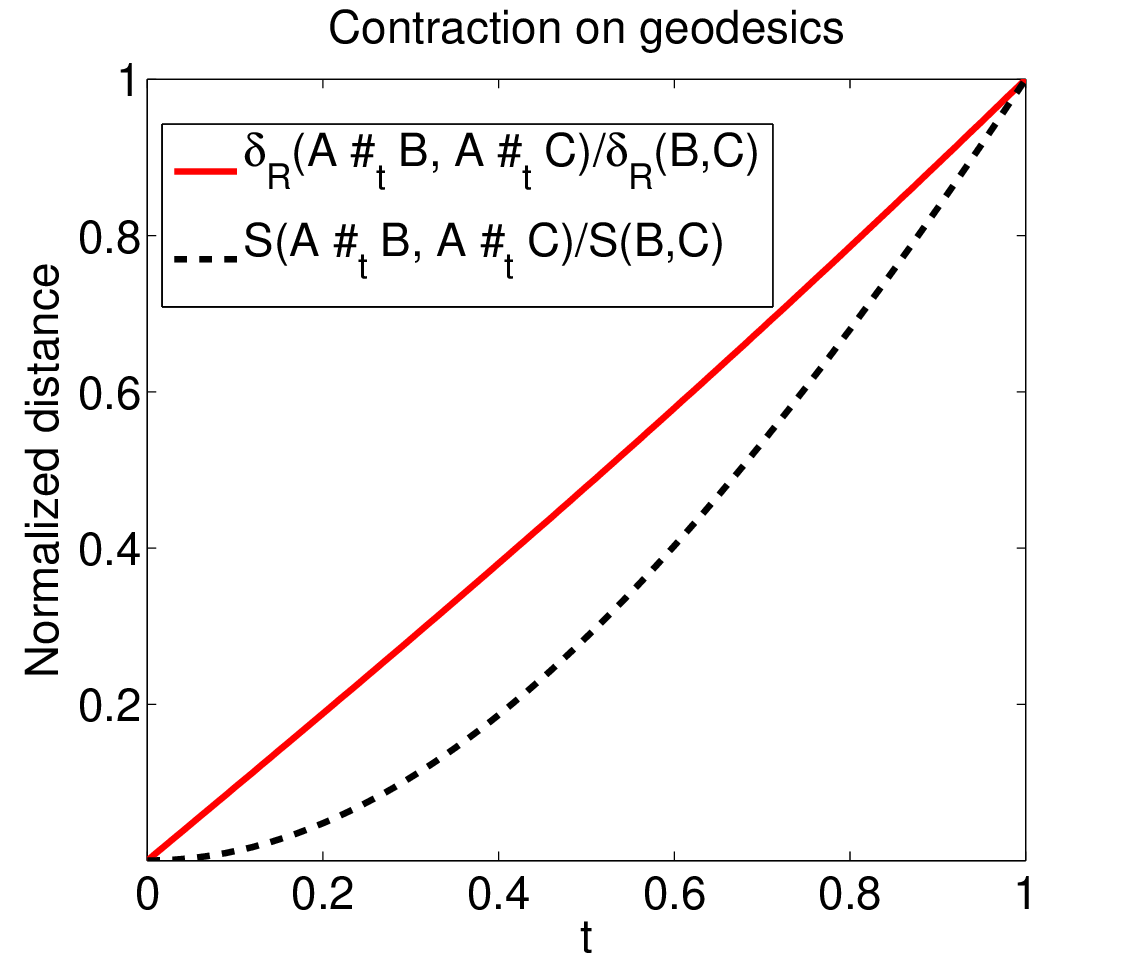}
  \caption{\textsf{Left to right:} Illustration of Theorems~\ref{thm.contract}, \ref{thm.contract3} and \ref{thm.cancel}. As expected, the \sdiv (denoted $S$ in the plots) exhibits slightly stronger contraction than $\riem$. Interestingly, the curves for $\riem$ are almost straight lines even for theorems~\ref{thm.contract} and~\ref{thm.cancel}, while those of $S$ have a more complicated shape; empirically, for random $A$ and $B$, the curves for $S$ are fit fairly well using a cubic in $t$.}
  \label{fig.pow}
\end{figure}

The GM $A \gm B$ is the midpoint $\gamma(1/2)$ on the curve~\eqref{eq.46}; an arbitrary point on this geodesic is therefore, frequently written as
\begin{equation}
  \label{eq.50}
  A \gm_t B := A^{1/2}(A^{-1/2}BA^{-1/2})^tA^{1/2}\quad\text{for}\ t \in [0,1].
\end{equation}
On this geodesics $\riem$ satisfies the following ``cancellation'' inequality~\cite[Thm.~6.1.12]{bhatia07}:
\begin{equation}
  \label{eq.49}
  \riem(A \gm_t B, A \gm_t C) \le t\riem(B, C)\quad\text{for}\ A, B, C > 0,\quad\text{and}\ t \in [0,1].
\end{equation}
We show that a similar inequality holds for $S$.

\begin{theorem}
  \label{thm.cancel}
  Let $A, B, C > 0$, and $t \in [0,1]$. Then, 
  \begin{equation}
    \label{eq.51}
    \ds^2(A \gm_t B, A \gm_t C) \le t\ds^2(B,C),\qquad t \in [0,1].
  \end{equation}
\end{theorem}
\begin{proof}
  Prop.~\ref{prop.basic} and Theorem~\ref{thm.contract} help prove this claim as follows:
  \begin{align*}
    \ds^2( A \gm_t B, A \gm_t C) &=
    \ds^2(A^{1/2}(A^{-1/2}BA^{-1/2})^tA^{1/2},A^{1/2}(A^{-1/2}CA^{-1/2})^tA^{1/2})\\
    &= \ds^2((A^{-1/2}BA^{-1/2})^t, (A^{-1/2}CA^{-1/2})^t)\\
    &\stackrel{\text{Thm.~\ref{thm.contract}}}{\le} t\ds^2(A^{-1/2}BA^{-1/2},A^{-1/2}CA^{-1/2}) = t\ds^2(B,C).\qed
  \end{align*}
\end{proof}

\subsubsection{A power-monotonicity property}
Above we saw that $\ds^2$ and $\riem$ show similar contractive behavior. Now we show that on matrix powers, they exhibit a similar monotonicity property (akin to a power-means inequality).
\begin{theorem}
  \label{thm.tucontract}
  Let $A, B > 0$. Let scalars $t$ and $u$ satisfy $1 \le t \le u < \infty$. Then,
  \begin{align}
    \label{eq.23}
    \inv{t}\riem(A^t,B^t) &\le \inv{u}\riem(A^u,B^u)\\
    \label{eq.21}
    \inv{t}\ssd^2(A^t,B^t) &\le \inv{u}\ssd^2(A^u,B^u).
  \end{align}
\end{theorem}
To our knowledge, even inequality~\eqref{eq.23} is new. Before proving Theorem~\ref{thm.tucontract} we first prove an auxiliary result (Prop.~\ref{prop.tudet}).

Let $x$ and $y$ be vectors in $\reals_+^n$. Denote by $z^\downarrow$ the vector obtained by arranging the elements of $z$ in decreasing order. 
We write,
\begin{align}
  \label{eq.15}
  &x \majlw y,\quad\text{if}\quad \nlprod_{j=1}^k x_j^\downarrow \le \nlprod_{j=1}^k y_j^\downarrow,\quad\text{for}\  1 \le k \le n;\quad\text{and}\\
  \label{eq.18}
  &x \majl y,\quad\text{if}\quad x \majlw y\quad\text{and}\ \nlprod_{j=1}^n x_j^\downarrow = \nlprod_{j=1}^n y_j^\downarrow.
\end{align}
Relation~\eqref{eq.15} is called \emph{weak log-majorization}, while~\eqref{eq.18} is known as \emph{log majorization} \cite[Ch.~2]{bhatia97}. Usual \emph{weak majorization} is denoted as:
\begin{equation}
  \label{eq.19}
  x \majw y\quad\text{if}\quad \nlsum_{i=1}^k x_j^\downarrow \le
  \nlsum_{j=1}^k y_j^\downarrow,\quad\text{for}\  1 \le k \le n.
\end{equation}

We now state a simple ``power-means'' determinantal inequality (which also follows from a more general monotonicity theorem of~\cite{bhaSub78} on power means). 
\begin{proposition}
  \label{prop.tudet}
  Let $A, B > 0$; let scalars $t$, $u$ satisfy $1 \le t \le u < \infty$. Then,
  \begin{equation}
    \label{eq.3}
    \mydet^{1/t}\Bigl(\tfrac{A^t+B^t}{2}\Bigr)\quad\le\quad\mydet^{1/u}\Bigl(\tfrac{A^u+B^u}{2}\Bigr).
  \end{equation}
\end{proposition}
\begin{proof}
  Let $P=\inv{A}$, and $Q=B$. To show~\eqref{eq.3}, we may equivalently show that
  \begin{equation}
    \label{eq.17}
    \nlprod_{j=1}^n\left(\tfrac{1+\lambda_j(P^tQ^t)}{2}\right)^{1/t} \le
    \nlprod_{j=1}^n\left(\tfrac{1+\lambda_j(P^uQ^u)}{2}\right)^{1/u}.
  \end{equation}
  Now recall the log-majorization~\cite[Theorem~IX.2.9]{bhatia97}:
  \begin{equation}
    \label{eq.6}
    \lambda^{1/t}(P^tQ^t) \majl \lambda^{1/u}(P^uQ^u),
  \end{equation}
  and apply to it the monotonic function $f(r) = \log(1+r^u)$ to obtain the inequalities
  \begin{equation*}
    \nlsum_{j=1}^k \log(1+\lambda_j^{u/t}(P^tQ^t)) \le \nlsum_{j=1}^k \log(1+\lambda_j(P^tQ^t)),\quad 1\le k\le n.
  \end{equation*}
  Since log and $r \mapsto r^{1/u}$ are monotonic functions, these inequalities imply that
  \begin{equation*}
    \nlprod_{j=1}^k\Bigl(\tfrac{1+\lambda_j^{u/t}(P^tQ^t)}{2}\Bigr)^{1/u}
    \le\quad\nlprod_{j=1}^k\Bigl(\tfrac{1+\lambda_j(P^uQ^u)}{2}\Bigr)^{1/u}\quad 1 \le k \le n.
  \end{equation*}
  But since $u \ge t$, the function $r \mapsto r^{u/t}$ is convex. Thus,
  \begin{equation*}
    \begin{split}
      \nlprod_{j=1}^k\left(\tfrac{1+\lambda_j^{u/t}(P^tQ^t)}{2}\right)^{1/u}
      &\ge\quad
      \nlprod_{j=1}^k\Bigl[\left(\tfrac{1+\lambda_j(P^tQ^t)}{2}\right)^{u/t}\Bigr]^{1/u}\\
      &=
      \nlprod_{j=1}^k\left(\tfrac{1+\lambda_j(P^tQ^t)}{2}\right)^{1/t}.~\hskip1cm\qed
    \end{split}
  \end{equation*}
\end{proof}

\noindent\emph{Proof}~(Theorem~\ref{thm.tucontract}).\\
Part~(i): First observe that $\riem(X,Y)= \frob{\log E^\da(X\inv{Y})}$. Thus, we need to show
\begin{equation*}
  \tfrac{1}{t}\frob{\log E^\da(A^tB^{-t})} \le \tfrac{1}{u}\frob{\log E^\da(A^uB^{-u})}.
\end{equation*}
Equivalently, for vectors of eigenvalues we may prove
\begin{equation}
  \label{eq.4}
  \pnorm{\log\lambda^{1/t}(A^tB^{-t})}{2} \le \pnorm{\log\lambda^{1/u}(A^uB^{-u})}{2}.
\end{equation}
The log-majorization~\eqref{eq.6} yields the majorization inequality
\begin{equation*}
  \log \lambda^{1/t}(A^tB^{-t}) \prec \log \lambda^{1/u}(A^uB^{-u}),
\end{equation*}
to which we apply the map $x \mapsto \pnorm{x}{2}$ immediately obtaining~\eqref{eq.4}. Notice, that we have in fact proved the more general result
\begin{equation*}
  \tfrac{1}{t}\mynorm{\log E^\da(A^tB^{-t})}{\Phi} \le   \tfrac{1}{u}\mynorm{\log E^\da(A^uB^{-u})}{\Phi},
\end{equation*}
where $\Phi$ is a symmetric gauge function (a permutation invariant absolute norm).

Part~(ii): To prove~(\ref{eq.21}) we must show that
\begin{equation*}
  \tfrac{1}{t}\log\det\bigl( (A^t+B^t)/2\bigr) - \tfrac{t}{2}\log\det(A^tB^t) \le
  \tfrac{1}{u}\log\det\bigl( (A^u+B^u)/2\bigr) - \tfrac{u}{2}\log\det(A^uB^u).
\end{equation*}
But this inequality is immediate from Prop.~\ref{prop.tudet} and the monotonicity of log.\qed

\subsubsection{Contraction under translation}
The last basic contraction result that we prove is an analogue of the following important property~\cite[Prop.~1.6]{bougerol}:
\begin{equation}
  \label{eq.25}
  \riem(A+X, A+Y) \le \frac{\alpha}{\alpha+\beta}\riem(X,Y),\quad\text{for}\ A \ge 0,\ \text{and}\ X, Y > 0,
\end{equation}
where $\alpha=\max\set{\enorm{X},\enorm{Y}}$ and $\beta=\lambda_{\min}(A)$. This result plays a key role in deriving contractive maps for solving certain nonlinear matrix equations~\cite{leeLim}.  

We show a similar result for the \sdiv.
\begin{theorem}
  \label{thm.contract2}
  Let $X, Y > 0$, and $A \ge 0$, then
  \begin{equation}
    \label{eq.27}
    g(A) := \ds^2(A+X,A+Y),
  \end{equation}
  is monotonically decreasing and convex in $A$.
\end{theorem}
\begin{proof}
  We wish to show that if $A \le B$, then $g(A) \ge g(B)$. Equivalently, we
  can show that the gradient $\nabla_Ag(A) \le 0$~\cite[Section~3.6]{boyd}; to that end, we compute
  \begin{equation*}
    \nabla_A g(A) = \invp{\tfrac{(A+X) + (A+Y)}{2}} - \tfrac{1}{2}\invp{A+X}-\tfrac{1}{2}\invp{A+Y},
  \end{equation*}
  which is easily seen to be negative since $X \mapsto \inv{X}$ is operator convex.

  To prove that $g$ is convex, we look at its Hessian $\nabla^2g(A)$. Using the shorthand $P=\invp{A+X}$ and $Q=\invp{B+X}$, this Hessian is seen to be
  \begin{equation*}
    \nabla^2g(A) = \half(P \kron P + Q \kron Q) - \invp{\tfrac{\inv{P} + \inv{Q}}{2}} \kron \invp{\tfrac{\inv{P} + \inv{Q}}{2}}.
  \end{equation*}
  Again using operator convexity of $X \mapsto X^{-1}$ we obtain 
  \begin{equation*}
    \nabla^2g(A) \ge \tfrac{P \kron P + Q \kron Q}{2} - \tfrac{P+Q}{2} \kron \tfrac{P+Q}{2},
  \end{equation*}
  which is easily seen to be semidefinite because $P \ge Q$ and 
  \begin{equation*}
    P \kron P + Q \kron Q - P \kron Q + Q \kron P = (P-Q)\kron (P-Q) \ge 0.\hskip1cm\qed
  \end{equation*}
\end{proof}
The following corollary is immediate (\emph{cf.}~(\ref{eq.25})).
\begin{corollary}
  Let $X, Y > 0$, $A \ge 0$, $\beta = \lambda_{\min}(A)$. Then,
  \begin{equation}
    \label{eq.28}
    \ds^2(A+X, A+Y) \le \ds^2(\beta I + X, \beta I + Y) \le \ds^2(X, Y).
  \end{equation}
\end{corollary}
\begin{figure}[tbp]
  \centering
  \includegraphics[scale=0.27]{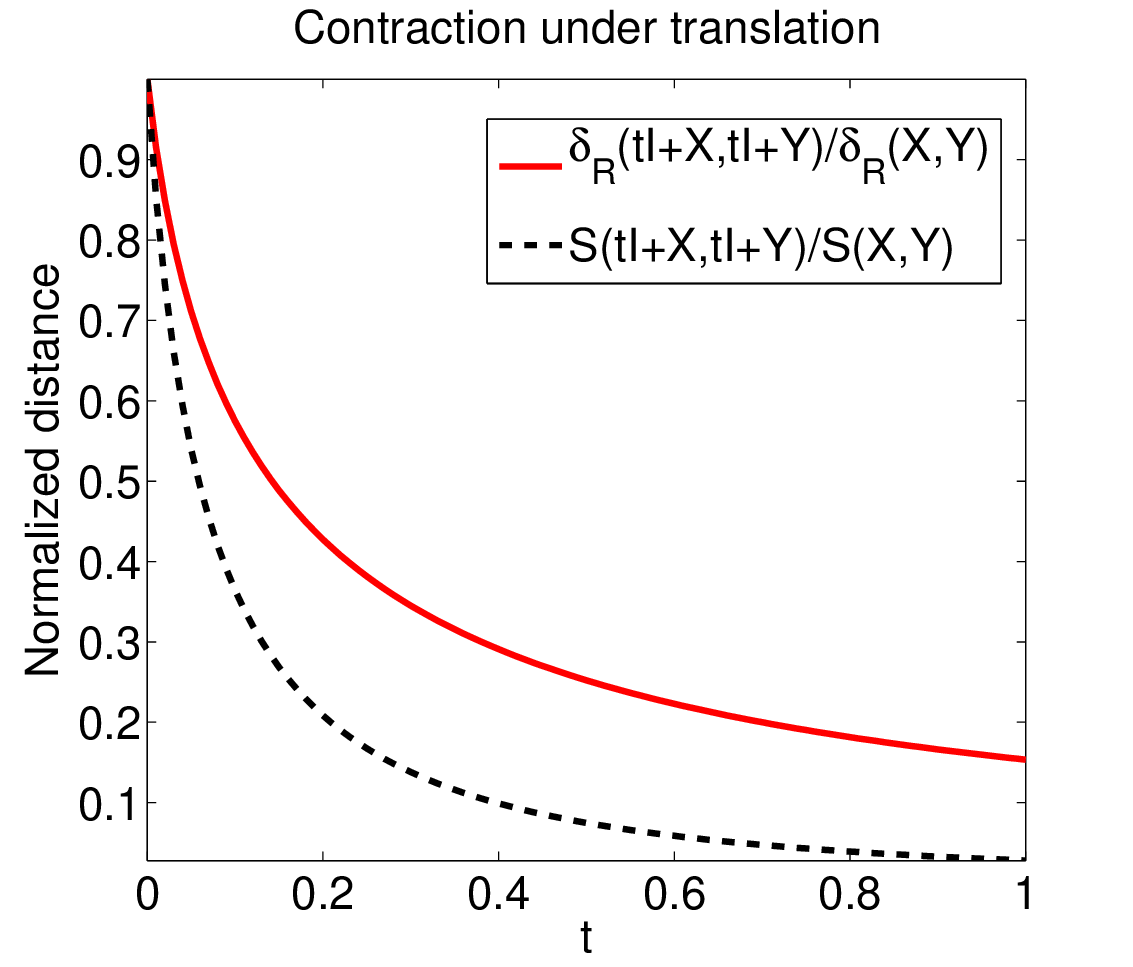}
  \caption{Illustration of Thm.~\ref{thm.contract2}. The plot shows the amount of contraction displayed by $\riem$ and $\ds^2$, for a pair of matrices $X$ and $Y$ with eigenvalues in $(0,1)$, when translated by $tI$ ($t \in [0,1]$). We see that $S$ is more contractive than $\riem$; more interestingly, the shape of the two curves is similar.}
  \label{fig.trans}
\end{figure}

\vspace*{-15pt}
\subsection{Contraction under compression}
In this section we establish a powerful compression property of the S-divergence, which connects it intimately with the Hilbert and Thompson metrics. We begin by setting up a few key results.

\begin{proposition}
  \label{prop.logdet}
  Let $P \in \C^{n\times k}$ ($k \le n$) have full column rank. The function $f : \pd_n \to \reals \equiv X \mapsto \log\det(P^*XP)-\log\det(X)$ is operator decreasing. 
\end{proposition}
\begin{proof}
  It suffices to show that $\nabla f(X) \le 0$. This amounts to establishing that
  \begin{equation}
    \label{eq:2}
    P(P^*XP)^{-1}P^* \le \inv{X}\quad\Leftrightarrow\quad
    \begin{bmatrix}
      \inv{X} & P\\
      P^*     & P^*XP
    \end{bmatrix} \ge 0.
  \end{equation}
  Since {\scriptsize$\begin{bmatrix} \inv{X} & I\\ I & X\end{bmatrix}$} $\ge 0$, the inequality~\eqref{eq:2} follows immediately upon realizing that
  \begin{equation*}
    \begin{bmatrix}
      \inv{X} & P\\
      P^*     & P^*XP
    \end{bmatrix} =
    \begin{bmatrix}
      I & 0\\
      0 & P^*
    \end{bmatrix}
    \begin{bmatrix}
      \inv{X} & I\\
      I       & X
    \end{bmatrix}
    \begin{bmatrix}
      I & 0\\
      0 & P
    \end{bmatrix}.\qed
  \end{equation*}
\end{proof}

\begin{corollary}
  \label{cor.ratio}
  Let $X, Y > 0$. Let $A=\pfrac{X+Y}{2}$, $G=X\gm Y$; and let $P\in \C^{n\times k}$ ($k \le n$) have full column rank. Then,
  \begin{equation}
    \label{eq.16}
    \frac{\det(P^*AP)}{\det(P^*GP)} \le \frac{\det(A)}{\det(G)}.
  \end{equation}
\end{corollary}
\begin{proof}
  Since $A \ge G$, it follows from Prop.~\ref{prop.logdet} that
  \begin{equation*}
    \log\det(P^*AP)-\log\det(A) \le \log\det(P^*GP)-\log\det(G).
  \end{equation*}
  Rearranging, and using the fact that $P^*AP \ge P^*GP$, we obtain~\eqref{eq.16}.
\end{proof}

\begin{theorem}[\!\!\protect{\citep[Thm.~3]{ando79}}]
  \label{prop.linmap}
  Let $\Pi : \pd_n \to \pd_k$ be a positive linear map. Then,
  \begin{equation}
    \label{eq.24}
    \Pi(A\gm B) \le \Pi(A)\gm \Pi(B)\qquad\text{for } A, B \in \pd_n.
  \end{equation}
\end{theorem}

We are now ready to prove the main theorem of this section.
\begin{theorem}
  \label{thm.compr}
  Let $P\in \C^{n\times k}$ ($k \le n$) have full column rank. Then,
  \begin{equation}
    \label{eq.9}
    \ds^2(P^*AP,P^*BP) \le \ds^2(A,B)\qquad\text{for } A, B \in \pd_n.
  \end{equation}
\end{theorem}
\begin{proof}
  Observe that~\eqref{eq.9} does not follow from the known inequality $f(U^*AU) \le U^*f(A)U$ for operator convex $f$, because $\ds^2$ is nonconvex. We need to show that
  \begin{equation}
    \label{eq.39}
    \log\frac{\det\pfrac{P^*(A+B)P}{2}}{\sqrt{\det(P^*AP)\det(P^*BP)}} \le \log\frac{\det\pfrac{A+B}{2}}{\sqrt{\det(AB)}}.
  \end{equation}
  From Prop.~\ref{prop.linmap} it follows that $P^*(A\gm B)P \le (P^*AP)\gm(P^*BP)$, which implies that
  \begin{equation*}
    \frac{1}{\sqrt{\det(P^*AP)\det(P^*BP)}} = \frac{1}{\det[(P^*AP)\gm(P^*BP)]} \le \frac{1}{\det(P^*(A\gm B)P)}.
  \end{equation*}
  Invoking Corollary~\ref{cor.ratio} and taking logarithms we obtain~\eqref{eq.39}.\qed
\end{proof}

\begin{corollary}
  \label{cor.compr.tensor}
  Let $A, B, C, D > 0$; let $\circ$ denote the Hadamard product. Then,
  \begin{equation}
    \label{eq.40}
    \ds^2(A\circ B, C \circ D) \le \ds^2(A \kron B, C \kron D).
  \end{equation}
\end{corollary}
\begin{proof}
  We know that $A\circ B$ is a principal submatrix of $A \kron B$. In particular, there is a projection $P$ such that $P^*(A\kron B)P = A\circ B$. So, \eqref{eq.40} reduces to showing that $\ds^2(P^*(A\kron B)P, P^*(C\kron D)P) \le \ds^2(A \kron B, C \kron D)$, which follows from Thm.~\ref{thm.compr}.\qed
\end{proof}

One may wonder if Theorem~\ref{thm.compr} holds more generally for all positive linear maps, not just congruence transforms. The answer turns out to be negative, as may be seen by considering $\Pi : X \mapsto X \oplus X$. Then, $\ds^2(\Pi(X),\Pi(Y)) =2\ds^2(X,Y) \not\le \ds^2(X,Y)$.

Next, one may ask whether a Theorem~\ref{thm.compr} extends to $\riem$? Corollary~\ref{corr.riem.compress} shows that this compression does extend to $\riem$, and actually follows from a more general theorem (Theorem~\ref{thm.geig}); we believe that this basic theorem must exist in the literature, but provide our own proof for completeness.

\begin{theorem}
  \label{thm.geig}
  Let $A, B \in \pd_n$, and $P\in \C^{n\times k}$ ($k \le n$) have full colrank. Then,
  \begin{equation}
    \label{eq.66}
    \lambda_j^\da(P^*AP,P^*BP) \le \lambda_j^\da(A,B),\quad 1 \le j \le k.
  \end{equation}
\end{theorem}
\begin{proof}
  Since $B$ positive definite, the eigenvalue min-max theorem shows that 
  \begin{equation*}
    \lambda_j^\da(A,B) = \min_{\dim V = j}\max_{x \in V}\frac{x^*Ax}{x^*Bx}.
  \end{equation*}
  From this variational representation it follows that
  \begin{equation*}
    \begin{split}
      \lambda_j^\da(A,B)\quad&=\quad\min_{\dim V = j}\max_{w \in V}\frac{w^*Aw}{w^*Bw} \quad\ge\quad\min_{\dim V = j}\ \max_{\substack{w\\ w=Px, w \in V, x \neq 0}}\frac{w^*Aw}{w^*Bw}\\
      &=\quad\min_{\dim V=j}\max_{x \in V}\frac{x^*P^*APx}{x^*P^*BPx}=\lambda_j^\da(P^*AP,P^*BP).
    \end{split}
  \end{equation*}
  The second-to-last equality above holds since $\set{Px | x \neq 0}$ is a subspace of dimension $j$, due to $P$ having full column rank.\qed
\end{proof}

\begin{corollary}
  \label{corr.riem.compress}
  Let $P \in \C^{n\times k}$ ($k \le n$) have full column rank. Then,
  \begin{equation}
    \label{eq.57}
    \delta_\Phi(P^*AP,P^*BP) \le \delta_\Phi(A,B),
  \end{equation}
  where $\Phi$ is any symmetric gauge function.
\end{corollary}
\begin{proof}
  Recall that $\delta_\Phi(A,B)=\pnorm{\log E^\da(A^{-1}B)}{\Phi}$. Thus, our task is to show that
  \begin{equation}
    \label{eq.65}
    \pnorm{\log E^\da(P^*AP\invp{P^*BP})}{\Phi} \le \delta_\Phi(A,B).
  \end{equation}
  This result follows by realizing that $\lambda(A^{-1}B)=\lambda(A,B)$ and invoking Theorem~\ref{thm.geig}.
\end{proof}

\subsection{Differences between $\ds$ and $\riem$}
So far we have highlighted similarities between $\ds$ and $\riem$. It is worthwhile to highlight some differences too. Since we have implicitly already covered this ground, we summarize these differences in Table~\ref{tab.diff}.

\begin{table}[h]\footnotesize
  \begin{tabular}{l|l||l|l}
    Riemannian metric & Ref. & \sdiv & Ref.\\
    \hline
    
    Eigenvalue computations needed & E & Cholesky decompositions suffice & E \\
    $e^{-\beta\riem^2(X,Y)}$ usually not a kernel & E & $e^{-\beta\ds^2(X,Y)}$ a kernel for many $\beta$ &Th.\ref{thm.wallach}\\
    $\riem$ geodesically convex & \citep[Th.6]{bhatia07} & $\ssd$ not geodesically convex & E\\
    $(\pd_n,\riem)$ is a CAT(0)-space & \citep[Ch.6]{bhatia07} & $(\pd_n,\ssd)$ not a CAT(0)-space & E\\
    Computing means with $\riem^2$ difficult & \citep{bruno} & Computing means with $\ds^2$ easier & Sec.~\ref{sec.gm}\\
  \end{tabular}
  \caption{Some differences between $\riem$ and $\ds$ at a glance. An `E' indicates that it is easy to verify the claim or to find a counterexample.}\label{tab.diff}
\end{table}

\subsection{Bi-Lipschitz-like comparison}
We end our discussion of relations between $\riem$ and $\ds$ by showing how they directly compare with each other; here, our main result is the sandwiching inequality~\eqref{eq.54}.
\begin{theorem}
  \label{thm.bnds}
  Let $A, B \in \pd_n$. 
  Then, we have the following bounds
  \begin{equation}
    \label{eq.54}
    8\ds^2(A,B) \le \riem^2(A,B) \le 2\thom(A,B)\bigl(\ds^2(A,B) + n\log 2\bigr).
  \end{equation}
\end{theorem}
\begin{proof}
  First we establish the upper bound. To that end, we first rewrite $\riem$ as
  \begin{equation}
    \label{eq.5}
    \riem(A,B) := \bigl(\nlsum_i \log^2 \lambda_i(A\inv{B})\bigr)^{1/2}.
  \end{equation}
  Since  $\lambda_i(A\inv{B}) > 0$, we may write $\lambda_i(A\inv{B}) := e^{u_i}$ for some $u_i$, whereby
  \begin{equation}
    \label{eq.52}
    \riem(A,B)=\enorm{u}\quad\text{and}\quad\thom(A,B) = \infnorm{u}.
  \end{equation}
  Using the same notation we also obtain 
  \begin{equation}
    \label{eq.7}
    \ds^2(A,B) = \nlsum_i (\log(1+e^{u_i}) - u_i/2 - \log 2).
  \end{equation}
  To relate the quantities~\eqref{eq.52} and~\eqref{eq.7}, it is helpful to consider the function
  \begin{equation*}
    f(u) := \log(1+e^u) - u/2 - \log 2.
  \end{equation*}
  If $u < 0$, then $\log(1+e^u) \ge \log 1 = 0$ holds and $-u/2=|u|/2$; while if $u \ge 0$, then $\log(1+e^u) \ge \log e^u = u$ holds. For both cases, we have the inequality
  \begin{equation}
    \label{eq.13}
    f(u) \ge |u|/2 - \log 2.
  \end{equation}
  Since $\ds^2(A,B)=\nlsum_i f(u_i)$, inequality~\eqref{eq.13} leads to the bound
  \begin{equation}
    \label{eq.8}
    \ds^2(A,B) \ge -n\log 2 + \half\nlsum_i |u_i| = \half \norml{u} -n\log 2.
  \end{equation}
  From H\"older's inequality we know that $u^Tu \le
  \infnorm{u}\norml{u}$; so we immediately obtain
  \begin{equation*}
    \riem^2(A,B) \le 2\thom(A,B)(\ds^2(A,B)+n\log 2).
  \end{equation*}
  To obtain the lower bound, consider the function
  \begin{equation}
    \label{eq.56}
    g(u, \sigma) := u^Tu - \sigma(\log(1+e^u) - u/2-\log 2).
  \end{equation}
  The first and second derivatives of $g$ with respect to $u$ are given by
  \begin{equation*}
    \label{eq:20}
    g'(u, \sigma) = 2u - \frac{\sigma e^u}{1+e^u} + \frac{\sigma}{2},\qquad
    g''(u, \sigma) = 2 - \frac{\sigma e^u}{(1+e^u)^2}.
  \end{equation*}
  Observe that for $u=0$, $g'(u, \sigma)=0$. To ensure that $0$ is the minimizer
  of~\eqref{eq.56}, we now determine the largest value of $\sigma$ for which
  $g'' \ge 0$. Write $z:=e^u$; we wish to ensure that $\sigma z/ (1+z)^2 \le 2$. Since $z \ge 0$, the arithmetic-geometric inequality shows that
  $\frac{z}{(1+z)^2} = \frac{\sqrt{z}}{1+z}\frac{\sqrt{z}}{1+z} \le
  \frac{1}{4}$. Thus, for $0 \le \sigma \le 8$, the
  inequality $\sigma z / (1+z)^2 \le 2$ holds (or equivalently $g''(u, \sigma) \ge
  0$). Hence, $0 = g(0, \sigma) \le g(u, \sigma)$, which implies that
  \begin{equation*}
    \riem^2(A,B) - \sigma\ds^2(A,B) = \nlsum_i g(u_i, \sigma) \ge 0,\quad\text{for}\ 0 \le \sigma \le 8.\hskip 1cm\qed
  \end{equation*}
\end{proof}

\vspace*{-12pt}
\section{\sdiv-mean}
\label{sec.mtxmeans}
We briefly mention the (nonconvex) problem of computing means for collection of input positive definite matrices. Similar conclusions (using different arguments) were previously obtained in~\citep{chebbi}---our analysis provides a complementary view.

Given input matrices $A_1,\ldots,A_m\in \pd_n$ and nonnegative weights $w_i \ge 0$ such that $\nlsum_{i=1}^m w_i = 1$, the \emph{S-mean} problem is to compute
\begin{equation}
  \label{eq.1}
  \min_{X > 0}\ h(X) := \nlsum_{i=1}^mw_i\ds^2(X,A_i),
\end{equation}
This problem was essentially studied in~\cite{iccv11}, and more thoroughly investigated by~\citet{chebbi}. Both~\citep{iccv11,chebbi} considered the necessary optimality condition (ignoring $X > 0$ for now)
\begin{equation}
  \label{eq.2}
  \nabla h(X) = 0\quad\Leftrightarrow\quad \inv{X} = \nlsum_i w_i\invp{\tfrac{X+A_i}{2}},
\end{equation}
and both made a minor oversight by claiming the unique positive definite solution to~\eqref{eq.2} to be the global minimum of~\eqref{eq.1}, whereas their proofs established only stationarity, neither global nor local optimality. This oversight is easily fixed.

Since $\ds^2$ is strictly geodesically convex (Theorem~\ref{thm.jointgc}), it follows that $h(X)$ is also strictly geodesically convex. Thus, once existence of a minimizer has been established, its  uniqueness is immediate---moreover,  ensuring~\eqref{eq.2} is also sufficient. Existence is also easy, because if $X \to 0$ or $X\to \infty$, the objective $h(X) \to \infty$.\footnote{We thank an anonymous referee for alerting us to the need of invoking the boundary behavior.}

\section{S-Divergence-median}
\label{sec.median}
Instead of minimizing a sum-of-squared distances, the \emph{geometric median} problem seeks a solution to
\begin{equation}
  \label{eq:1}
  \min_{X > 0}\quad \phi(X) := \nlsum_{i=1}^m w_i \ds(X,A_i).
\end{equation}
In some cases, geometric medians are more preferred than geometric means as they may be more robust~\citep{chaCheMoa13,nieBha13,barber}. The S-median~\eqref{eq:1} was recently also studied by~\citet{chaCheMoa13}, who used it for an application in diffusion tensor imaging. 

To solve~\eqref{eq:1} we make the simplifying assumption that the median $\neq A_i$ for any $1\le i\le m$---in particular, for all $X \neq A_i$, $\phi(X)$ is differentiable. As before ignoring the constraint $X> 0$, we then obtain the first-order necessary condition
\begin{equation}
  \label{eq:3}
  \nabla\phi(X) = \nlsum_{i=1}^m\frac{w_i}{\ds(X,A_i)}\left[\pfrac{X+A_i}{2}^{-1}-\inv{X} \right] = 0.
\end{equation}
For solving~\eqref{eq:3}, \citep{chaCheMoa13} propose the iterating the following nonlinear map\footnote{There is a typo in~\citep{chaCheMoa13}; the correct formula~\eqref{eq:3} is mentioned in~\citep[pg.19]{moakTalk}.}
\begin{equation}
  \label{eq:4}
  \Gc(X) = \sum_{i=1}^m \frac{w_i}{\ds(X,A_i)}
  \left[
    \sum_{i=1}^m \frac{w_i}{\ds(X,A_i)}\left(\frac{X+A_i}{2}\right)^{-1}
  \right]^{-1}.
\end{equation}
In~\citep{chaCheMoa13} the authors claim the iteration $X_{k+1}=\Gc(X_k)$ to be a  contraction under the Thompson metric, and use that claim to deduce existence, uniqueness, and construction of the S-median. Unfortunately, their contraction claim is erroneous. Indeed, 
\begin{equation*}
  \begin{split}
  X=
  \small\begin{bmatrix}
    10 &  3\\
     3 &  9
  \end{bmatrix},\quad &Y=
  \begin{bmatrix}
    8 &  -6\\
   -6 &   45
  \end{bmatrix},\quad A_1 =
  \begin{bmatrix}
    5 &  5\\
    5 &  10
  \end{bmatrix},\quad A_2=
  \begin{bmatrix}
    10 &  1\\
     1 &  5
  \end{bmatrix}, w_1=w_2=\frac12,\\
  \implies\quad&\quad\thom(\Gc(X),\Gc(Y)) = 2.9314 > \thom(X,Y)=1.8324,
\end{split}
\end{equation*}
which shows that $\Gc(X)$ is \emph{not} a contraction under the Thompson metric $\thom$.

Fortunately, the map~(\ref{eq:4}) still leads to a valid fixed-point iteration, but with a different choice of metric. Specifically, instead of $\thom$, we propose to use \emph{Hilbert's projective metric} on $\pd_n$ which is given by~\citep[see e.g.,]{lemNuss12,koufany}:
\begin{equation}
  \label{eq:5}
  \hilb(X,Y) := \log\left(\frac{\lambda_M(X,Y)}{\lambda_m(X,Y)}\right),
\end{equation}
where $\lambda_M$ ($\lambda_m$) denotes the largest (smallest) generalized eigenvalue of $(X,Y)$. 

To prove our main result (Theorem~\ref{thm.smedian.contract}) for this section we need to first recall the following key properties of $\hilb$.
\begin{proposition}
  \label{prop.hilb}
  The Hilbert projective metric $\hilb$ satisfies the following:
  \begin{enumerate}[(i)]
    \setlength{\itemsep}{0pt}
  \item $\hilb(X^{-1},Y^{-1})=\hilb(X,Y)$
  \item $\hilb(\alpha X, \beta Y)=\hilb(X,Y)$ for all $\alpha, \beta > 0$ and $X, Y > 0$.
  \item $\hilb(\nlsum_{i=1}^m a_iX_i, \nlsum_{i=1}^m b_i Y_i) \le \max_{1\le i \le m}\hilb(X_i,Y_i)$, for $a_i,b_i > 0$ and $X_i,Y_i > 0$.
  \item Let $A \ge 0$ and $X, Y > 0$; then, $\hilb(A+X,A+Y) \le \frac{\alpha}{\alpha+\beta}\hilb(X,Y)$, where $\alpha=\max(\norm{X},\norm{Y})$ and $\beta=\lambda_{\min}(A)$.
  \end{enumerate}
\end{proposition}
\begin{proof}
  Property~(i) is obvious; (ii) is well-known~\citep{koufany}; (iii) follows the same argument as for $\thom$ in~\citep[Lemma~10.1(iv)]{lawLim08}; and (iv) follows from \citep[Prop.~1.6]{bougerol}.
\end{proof}

\begin{theorem}
  \label{thm.smedian.contract}
  Let $A_i > 0$ ($1\le i \le m$), $w_i \ge 0$ with $\nlsum_{i=1}^m w_i = 1$. Let $\Gc : \pd_n \to \pd_n$ be the nonlinear map defined~\eqref{eq:4}. Then, $\Gc$ is nonexpansive in $\hilb$, i.e.,
  \begin{equation*}
    \hilb(\Gc(X),\Gc(Y)) \le \hilb(X,Y),\qquad X, Y > 0,
  \end{equation*}
  Moreover, $\Gc$ is contractive if not all $A_i$ are equal (and $X\neq Y$).
\end{theorem}
\begin{proof}
 The projective property (Prop.~\ref{prop.hilb}-(ii)) of $\hilb$ proves crucial for analyzing~\eqref{eq:4}. Indeed, write $\Gc(X)=\alpha(X)\Fc(X)$, where
  \begin{equation*}
    \alpha(X) := \nlsum_{i=1}^m \frac{w_i}{\ds(X,A_i)}\quad\text{and}\quad 
    \Fc(X) = \left[
    \sum_{i=1}^m \frac{w_i}{\ds(X,A_i)}\left(\frac{X+A_i}{2}\right)^{-1}
  \right]^{-1}.
  \end{equation*}
  Thus, we may perform the following calculations
  \begin{align*}
    \hilb(\Gc(X),\Gc(Y)) &= \hilb(\alpha(X)\Fc(X),\alpha(Y)\Fc(Y))\\
    &= \hilb(\Fc(X),\Fc(Y))\\
    &\le\hilb(\nlsum_{i=1}^m \tfrac{w_i}{\ds(X,A_i)}\pfrac{X+A_i}{2}^{-1},\nlsum_{i=1}^m \tfrac{w_i}{\ds(X,A_i)}\pfrac{Y+A_i}{2}^{-1})\\
    &\le\max_{1\le i \le m}\hilb(X_i,Y_i),\qquad X_i = X+A_i, Y_i = Y+A_i\\
    &=\max_{1\le i \le m}\hilb(X+A_i,Y+A_i)\\
    &\le \max_{1\le i \le m}\gamma_i\hilb(X,Y),\qquad \gamma_i < 1\\
    &= \gamma\hilb(X,Y),
  \end{align*}
  where $\gamma_i := $ and $\gamma := \max_{1\le i \le m}\gamma_i$. The second inequality above uses Prop.~\ref{prop.hilb}-(ii),(iii), while the third one invokes Prop.~\ref{prop.hilb}-(iv). Clearly, if $X\neq Y$ and not all $A_i$ are equal (in which case the median is just $A_1$), $\Gc$ is a strict contraction.\qed
\end{proof}
\begin{corollary}
  \label{sec.corr}
  Starting with a suitable $X_0 > 0$, let $\set{X_k}_{k \ge 0}$ be the sequence generated by $X_{k+1} = \Gc(X_k)$. Assume that none of the $A_i$s is the median, and that $\ds(X_k,A_i) > 0$ for all $k$ and $i$. Then, the sequence $\set{X_k}$ converges to a point $X_*$ that is the unique positive definite solution to~\eqref{eq:3}, and this point is the S-median. 
\end{corollary}
\begin{proof}
  Observe that $\phi(0)=\phi(X)=+\infty$; thus, since $\phi(X)$ is continuous on $\pd_n$ it must attain its minimum in the interior. Thus, \eqref{eq:3} must have a positive definite solution. The metric space $(\set{X\ge 0},\hilb)$ is complete~\citep{koufany}, and Theorem~\ref{thm.smedian.contract} shows that under our assumptions, $\Gc$ is a strict contraction in $\hilb$. Therefore, if  \eqref{eq:4} has a solution, then from the fixed-point theorem of Edelstein~\citep{edel62}, it follows that $\Gc$ generates iterates which stay within a compact set and converge to the unique fixed point of $\Gc$. This fixed-point is positive definite by construction, and satisfies~\eqref{eq:3}, whereby it is the desired S-median.\qed
\end{proof}

Figure~\ref{fig:smed} illustrates empirical behavior of the fixed-point (FP) iteration $X_{k+1}=\Gc(X_k)$ on three different collections of positive matrices. The plot compares the FP iteration against a manifold based conjugate gradient method by showing Frobenius norms of the gradients $\nabla \phi$ obtained as a function of running time. The FP iteration turns out to run remarkably faster than the manifold conjugate gradient method (taken from~\citep{manopt}). Both the compared methods are implemented in \textsc{Matlab} and the experiments were run on a personal laptop with a quadcore Intel i7-3520M (2.90Ghz) processor under the Ubuntu 12.10 operating system.

\begin{figure}[tbp]
  \centering
  \includegraphics[width=.35\textwidth]{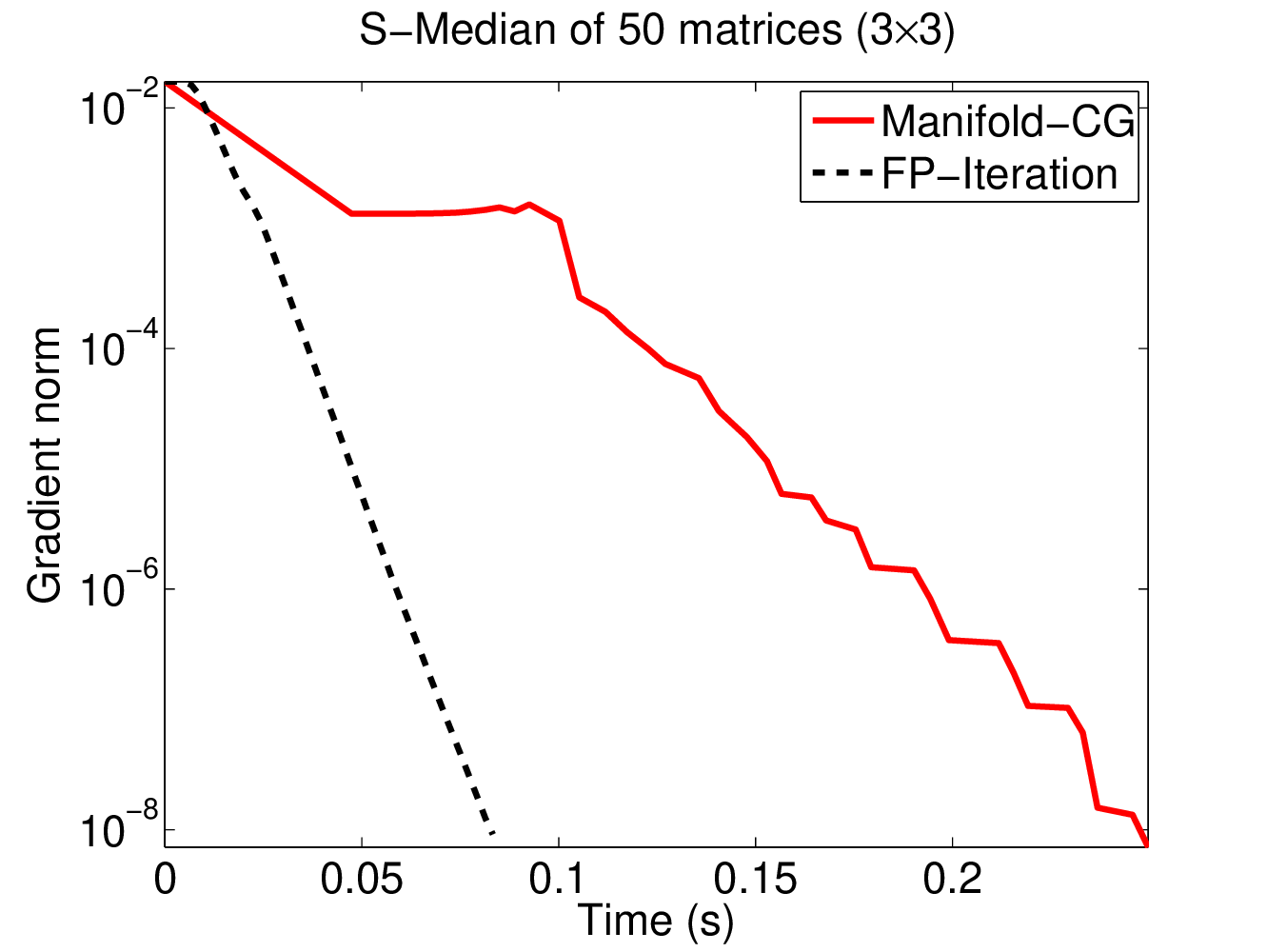}\hskip-10pt
  \includegraphics[width=.35\textwidth]{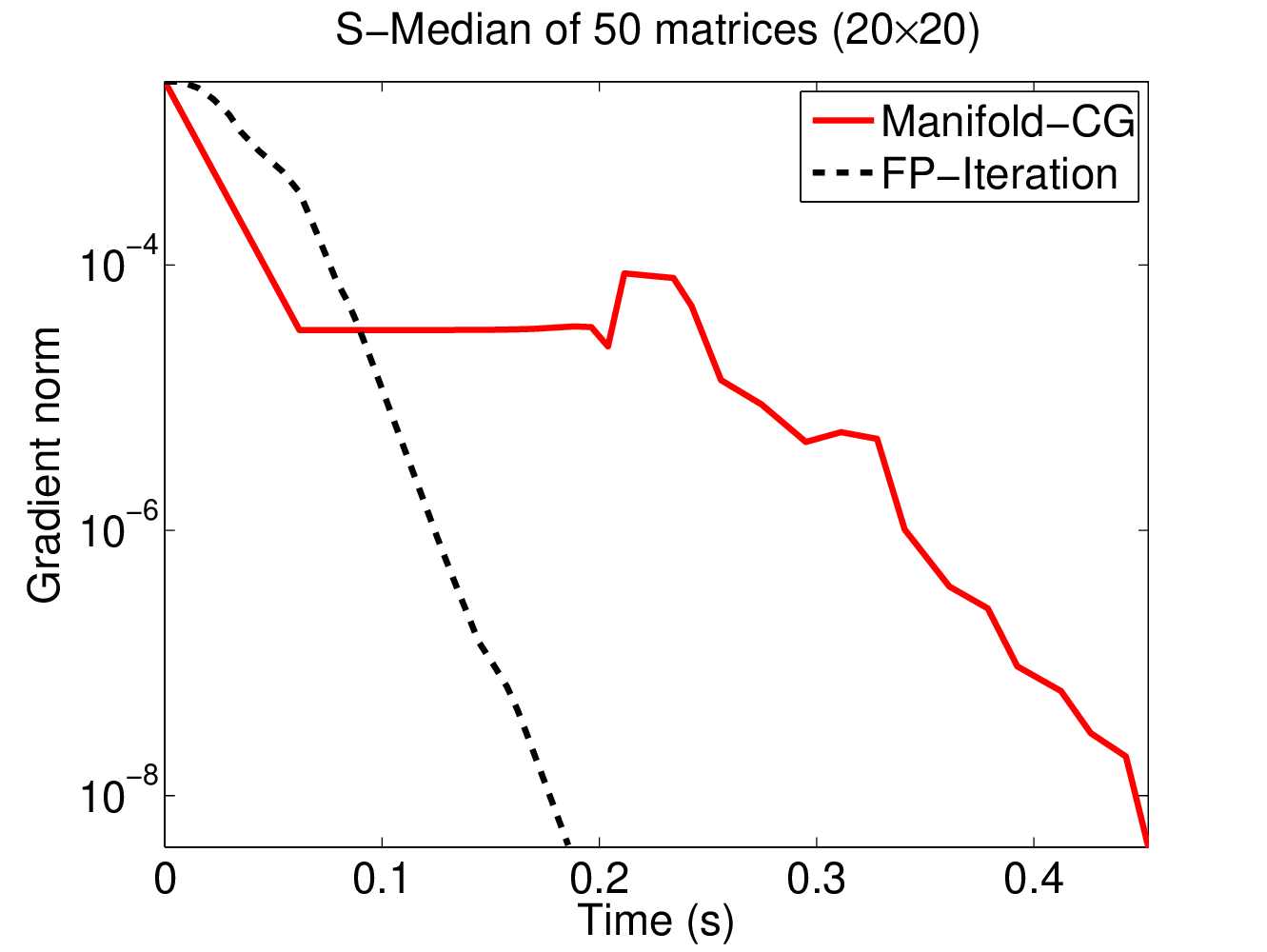}\hskip-10pt
  \includegraphics[width=.35\textwidth]{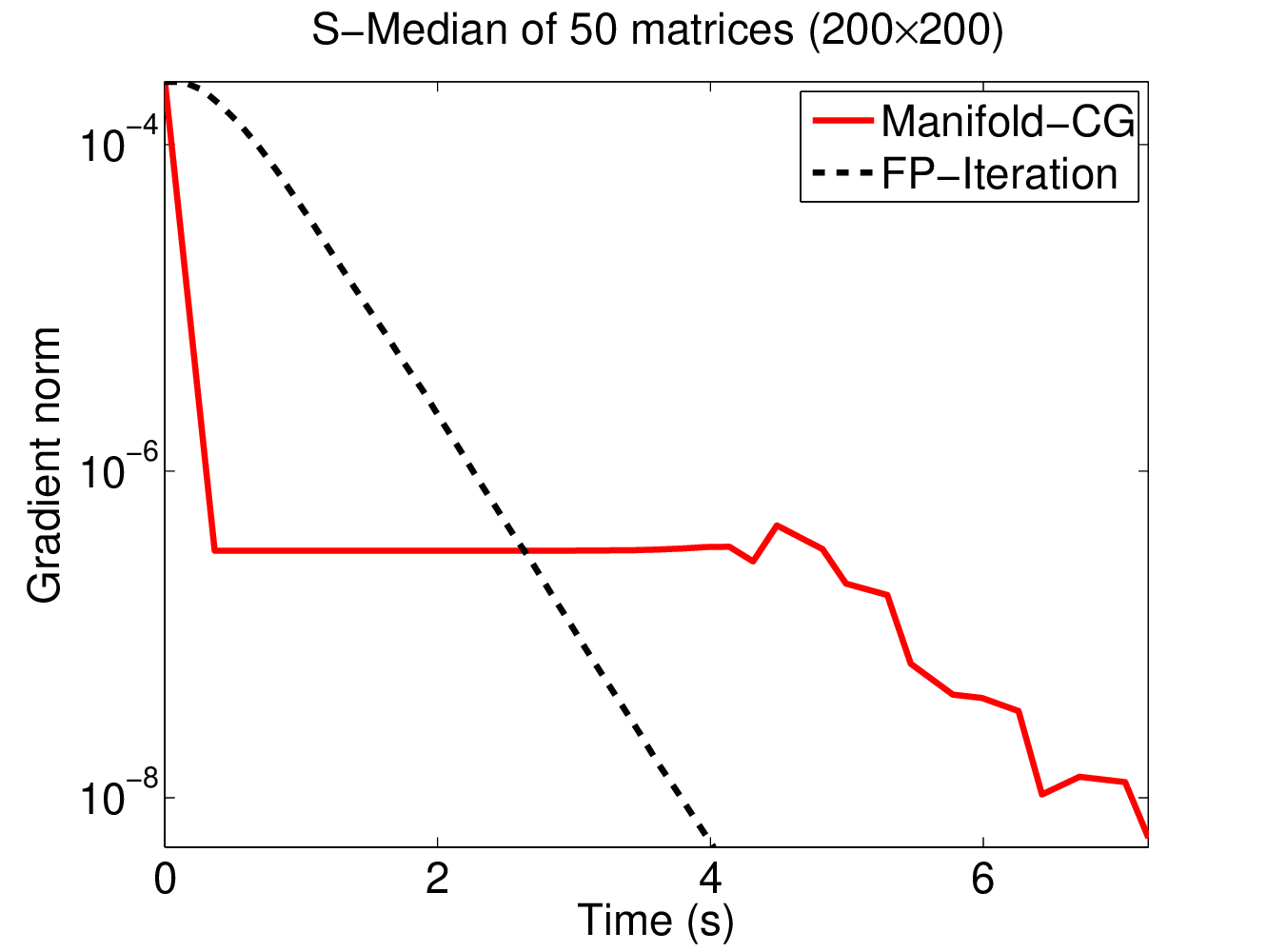}
  \caption{Convergence comparison between fixed-point iteration using~\eqref{eq:4} and the Riemannian conjugate-gradient method from \textsc{Manopt}~\citep{manopt}. The plots report the gradient-norm $\frob{\nabla \phi}$ as a function of running time (secs) for collections of random (Wishart) matrices in $\pd_3,\pd_{20}$, and $\pd_{200}$.}
  \label{fig:smed}
\end{figure}
\section{Discussion and Future Work}
\label{sec.dafw}
In this paper we studied the \sdiv (and its square-root) on positive definite matrices. We derived numerous results that uncovered qualitative similarities between the \sdiv and the Riemannian distance on the manifold of Hermitian positive definite matrices. Notably, we showed that the square root of the \sdiv actually defines a distance, albeit one that does not isometrically embed into any Hilbert space. As an application, we briefly discussed the problems of computing means and medians using the \sdiv, and provided a fixed-point algorithm for computing medians of positive definite matrices.

Several directions of future work are open. We mention some below.
\begin{itemize}
  \setlength{\itemsep}{0pt}
\item Deriving refinements of the main inequalities presented in this paper.
\item Studying properties of the metric space $(\pd^d,\ssd)$
\item Characterizing the subclass $\Xc \subset \pd^d$ of positive matrices for which $(\Xc, \ssd)$ admits an isometric Hilbert space embedding.
\item Developing better algorithms to compute the S-mean and the S-median.
\item Identifying more applications where $\ds^2$ (or $\ssd$) can be useful.
\end{itemize}

We hope that our paper encourages other researchers to investigate new properties and applications of the $S$-Divergence.

\subsection*{Acknowledgments}
I am grateful to Jeff Bilmes for hosting me at the EE Department at the University of Washington, during my unforeseen visit in July 2011. It was there where I first found the proof of Theorem~\ref{thm.metric}. 

\bibliographystyle{siam}

\end{document}